\documentclass[a4paper,reqno]{amsart}

\usepackage{amsmath,amsfonts,amsthm,amssymb,hyperref,latexsym,mathrsfs,enumerate,pgfplots,tikz-cd,tkz-graph,mathtools}

\usetikzlibrary{decorations.markings,automata,arrows,positioning,calc}

\makeatletter
\@namedef{subjclassname@2020}{%
  \textup{2020} Mathematics Subject Classification}
\makeatother

\newtheorem {theorem}{Theorem}[section]

\newtheorem {proposition}[theorem]{Proposition}
\newtheorem {corollary}[theorem]{Corollary}

\newtheorem {thm}{Theorem}

\theoremstyle {definition}

\newtheorem {example}[theorem]{Example}
\newtheorem {definition}[theorem]{Definition}
\newtheorem {remark}[theorem]{Remark}
\newtheorem {question}[theorem]{Question}

\numberwithin{equation}{section}

\title{Applications of $\mathrm{C}^*$-classification}

\author[B.~Jacelon]{Bhishan Jacelon}
\address[B.~Jacelon]{
Institute of Mathematics of the Czech Academy of Sciences\\ \v{Z}itn\'{a} 25\\115 67 Prague 1\\Czech Republic}
\email{jacelon@math.cas.cz}

\begin{document}

\subjclass[2020]{46L35, 46L85}
\keywords{Classification of nuclear $\mathrm{C}^*$-algebras, $K$-theory, unitary orbits, dynamical systems, Choquet simplices, Markov chains.}

\maketitle

\begin{abstract} 
We provide some background on the category of classifiable $\mathrm{C}^*$-algebras, whose objects are infinite-dimensional, simple, separable, unital $\mathrm{C}^*$-algebras that have finite nuclear dimension and satisfy the universal coefficient theorem, and describe some applications of the classification of objects in, and morphisms into, this category.
\end{abstract}

\section{Introduction} \label{section:intro}

The theme of classification is pervasive in the sciences and in mathematics, certainly in topology and unsurprisingly in operator algebras (even from the beginning \cite{Murray:1943le}). George Elliott parlayed the work of Glimm \cite{Glimm:1960qh} and Bratteli \cite{Bratteli:1972fj} into a classification of approximately finite-dimensional (AF) algebras by $K$-theory \cite{Elliott:1976kq}. Elliott conjectured that such a classification should hold in the much broader class of simple, separable, unital, nuclear $\mathrm{C}^*$-algebras. The question, in other words, was whether isomorphism of such algebras could be determined by isomorphism of their ordered $K$-groups, which in principle are simpler, computable mathematical objects. The project centred on proving this conjecture (with the invariant suitably fine tuned to include traces) came to be known as the Elliott programme, and it has been an active research area over the past few decades.

With the imposition of a regularity property called \emph{finite nuclear dimension} (an analogue of finite topological dimension) and modulo the universal coefficient theorem (UCT), which is a technical assumption on $K$-theory, the classification programme has proven a resounding success. The far-reaching theorem below is the culmination of the work of many mathematicians; see, for example, \cite{Kirchberg:2000kq,Phillips:2000fj,Elliott:2016ab,Gong:2020ud,Gong:2020uf,Carrion:wz,Winter:2012pi,Castillejos:2021wm,Tikuisis:aa}.

\begin{thm}[The classification theorem] \label{thm:class}
The class of infinite-dimensional, simple, separable, unital $\mathrm{C}^*$-algebras that have finite nuclear dimension and satisfy the UCT is classified by the Elliott invariant.
\end{thm}

We will provide some background on classification and the Elliott invariant in Sections~\ref{section:category}, \ref{section:ktheory} and \ref{section:traces}, but the primary aim of these notes is to demonstrate the usefulness of classification as an analytical tool. Actually, we will make crucial use not just of the classification of suitably well-behaved $\mathrm{C}^*$-algebras but also of morphisms into these algebras either by the Cuntz semigroup (see Section~\ref{section:traces}) or by a refinement of the Elliott invariant (see Section~\ref{section:chaos}). Armed with this, we will show how to tackle the Weyl problem (Section~\ref{section:weyl}), analyse the generic tracial behaviour of automorphisms (Section~\ref{section:chaos}) and randomly construct classifiable $\mathrm{C}^*$-algebras (Section~\ref{section:random}).

\section{The classifiable category} \label{section:category}

Theorem~\ref{thm:class} is a statement about a certain collection $\mathcal{E}$ of infinite-dimensional, simple, separable, unital $\mathrm{C}^*$-algebras.

\begin{definition} \label{def:category}
The \emph{classifiable category $\mathcal{E}$} is the category whose objects are infinite-dimensional, simple, separable, unital $\mathrm{C}^*$-algebras that have finite nuclear dimension and satisfy the UCT, and whose morphisms are unital $^*$-homomorphisms.
\end{definition}

The objects in $\mathcal{E}$ are sometimes referred to as \emph{classifiable $\mathrm{C}^*$-algebras}. They come in two flavours (see Definition~\ref{def:finite} and Remark~\ref{remark:finite}): \emph{purely infinite} like the Cuntz algebras $\mathcal{O}_n$ or \emph{stably finite} like the CAR algebra $M_{2^\infty}$. An example of a unital $\mathrm{C}^*$-algebra that is \emph{not} classifiable in this sense is $\mathbb{B}(\mathcal{H})$, for any Hilbert space $\mathcal{H}$. Indeed, if $\mathcal{H}$ is infinite dimensional then $\mathbb{B}(\mathcal{H})$ is nonsimple (the algebra $\mathbb{K}(\mathcal{H})$ of compact operators is a nontrivial closed, two-sided $^*$-ideal), nonnuclear (that is, it does not have the completely positive approximation property discussed in the proof of Proposition~\ref{prop:nuclear}) and indeed nonseparable (as a normed space). If on the other hand $\mathcal{H}$ is finite dimensional, then $\mathbb{B}(\mathcal{H})$ does satisfy most of the required hypotheses (including finite nuclear dimension and the universal coefficient theorem, which are discussed in Section~\ref{subsection:dimnuc} and Section~\ref{subsection:uct} respectively) but not infinite dimensionality (again as a normed space). This latter restriction is necessary given that $\mathcal{E}$ contains the Jiang--Su algebra $\mathcal{Z}$, which the Elliott invariant cannot distinguish from the complex numbers $\mathbb{C}$.

\subsection{The Elliott invariant}

\begin{definition}
The \emph{Elliott invariant} of a unital $\mathrm{C}^*$-algebra $A$ is
\[
\mathrm{Ell}(A) = (K_0(A),K_0(A)_+,[1_A],\mathrm{Aff}(T(A)),\rho_A\colon K_0(A)\to\mathrm{Aff}(T(A)),K_1(A)).
\]
\end{definition}

The components of the invariant are $K$-theory $K_*(A)$ (together with its order structure $K_0(A)_+$ and the $K_0$-class of the unit), the tracial state space $T(A)$, and the pairing $\rho_A$ between the two. They will be described in more detail in Section~\ref{section:ktheory} and Section~\ref{section:traces}. Although we will not precisely define the target category (whose morphisms consist of $K$-theoretic and tracial morphisms that are compatible with each other via the pairing map), we note that $\mathrm{Ell}$ is indeed a functor.

The precise statement of Theorem~\ref{thm:class} is that, if $A$ and $B$ are objects in $\mathcal{E}$, then there is an isomorphism of $\mathrm{C}^*$-algebras $\varphi\colon A\to B$ if and only if there is an isomorphism of invariants $\Phi\colon\mathrm{Ell}(A)\to\mathrm{Ell}(B)$ such that $\mathrm{Ell}(\varphi)=\Phi$. In fact, Theorem~\ref{thm:class} as presented in \cite{Carrion:wz} also entails a classification of \emph{morphisms} in $\mathcal{E}$, up to approximate unitary equivalence.

\begin{definition}
Two $^*$-homomorphisms $\varphi,\psi\colon A\to B$ between unital $\mathrm{C}^*$-algebras $A$ and $B$ are said to be \emph{approximately unitarily equivalent}, written $\varphi\sim_{au}\psi$, if for every $\varepsilon>0$ and every finite set $F\subseteq A$ there is a unitary $u\in B$ such that $\|\varphi(a)-u\psi(a)u^*\|<\varepsilon$ for every $a\in F$. 
\end{definition}

If $A$ is separable, this is equivalent to the existence of a sequence $(u_n)_{n\in\mathbb{N}}$ of unitaries in $B$ such that $\varphi(a)=\lim_{n\to\infty}u_n\psi(a)u_n^*$ for every $a\in A$.

Since $\mathrm{Ell}$ is invariant under $\sim_{au}$, agreement on $\mathrm{Ell}$ is a necessary condition for approximate unitary equivalence of $^*$-homomorphisms. But in general this is not sufficient. To classify $\sim_{au}$ equivalence classes of morphisms, we will rely on the Cuntz semigroup (see Section~\ref{section:traces}) or a suitable refinement of $\mathrm{Ell}$ that includes `total $K$-theory' and `Hausdorffised algebraic $K_1$' (see \cite{Carrion:wz} and Section~\ref{section:chaos}).

\subsection{Examples}

By \cite[Corollary 7.5]{Winter:2012pi} and \cite[Theorem A]{Castillejos:2021wm}, $\mathcal{E}$ contains every simple, separable, unital approximately subhomogeneous (ASH) algebra of slow dimension growth (in the sense of \cite[Definition 3.2]{Toms:2011fk}). These are sequential inductive limit $\mathrm{C}^*$-algebras $A=\varinjlim(A_n,\varphi_n)$ where:
\begin{itemize}
\item each $A_n$ is unital and subhomogeneous, that is, a subalgebra of $C(X_n,M_{k_n})$ for some compact metrisable space $X_n$ and $k_n\in\mathbb{N}$ (equivalently by a result of Blackadar \cite[Proposition 3.4.2]{Rordam:2002yu}, there is a finite bound on the dimensions of irreducible representations of $A_n$), and
\item each $\varphi\colon A_n\to A_{n+1}$ is a unital $^*$-homomorphism.
\end{itemize}

By \cite[Corollary 2.1]{Ng:2006kq}, the $A_n$ can be taken to be \emph{recursive} subhomogeneous, that is, described by iterated pullbacks of homogeneous algebras $M_{*}(C(X_{*}))$ (see \cite{Phillips:2001zm}). In this setting, `slow dimension growth' means that as $n$ tends to $\infty$, the ratios of the topological dimensions of these homogeneous components to the dimensions of the associated matrix fibres tends to $0$. In fact, by Theorem~\ref{thm:class} and \cite[Theorem 13.50]{Gong:2020ud}, \emph{every} stably finite object in $\mathcal{E}$ is isomorphic to an inductive limit of this form with \emph{no} dimension growth (the dimension of the base spaces being at most $3$), as such algebras exhaust the attainable range of the Elliott invariant for stably finite objects in $\mathcal{E}$. Dimension \emph{zero} corresponds to simple approximately finite-dimensional (AF) algebras like the CAR algebra $M_{2^\infty}$ and its uniformly hyperfinite (UHF) siblings. On the other hand, the `exotic' examples constructed in \cite{Villadsen:1998ys}, which led to the counterexamples \cite{Rordam:2003rz,Toms:2008vn} to the original Elliott conjecture, exhibit very fast dimension growth and indeed are not in $\mathcal{E}$.

Other examples of objects in $\mathcal{E}$, which do not \emph{a priori} have an inductive limit structure, are crossed product $\mathrm{C}^*$-algebras $C(X)\rtimes_\alpha\mathbb{Z}$ whenever $\alpha\colon C(X)\to C(X)$ is induced by a minimal homeomorphism of an infinite compact metrisable space $X$ whose Lebesgue covering dimension (described below) is finite. (Actually, classifiable crossed products are known to arise not just from single homeomorphisms but from a variety of group actions on compact metrisable spaces; see \cite{Kerr:2020ab,Kerr:2020aa,Gardella:2023aa,Gardella:2024aa}.) Minimality is needed for simplicity (see, for example, \cite[Theorem 9.5.8]{Strung:2021aa}) but finite nuclear dimension holds without this assumption (see \cite{Hirshberg:2017aa}). On the other hand, finite-dimensionality of $X$ cannot be dropped (see \cite{Giol:2010vs}).

Finally, $\mathcal{E}$ contains the $\mathrm{C}^*$-algebra $\mathrm{C}^*(E)$ associated to a finite graph $E$ without sinks whenever $\mathrm{C}^*(E)$ is simple. This includes, for example, the Cuntz algebras $\mathcal{O}_n$ (for $n<\infty$). In this case, $\mathrm{C}^*(E)$ is either AF or purely infinite (see \cite[Corollary 3.10]{Kumjian:1998nr} and \cite[Remark 5.6]{Bates:2000fk}). It should be noted that the classification of unital graph algebras by a $K$-theoretic invariant holds in full generality, i.e.\ without the assumption of simplicity (see \cite{Eilers:2021aa}).

\subsection{Finite nuclear dimension and $\mathcal{Z}$-stability} \label{subsection:dimnuc}

To motivate the definition of the nuclear dimension of a $\mathrm{C}^*$-algebra, let us recall the following. Note that, as pointed out in \cite[Remark 2.11]{Winter:2010dn}, we do not lose any generality by restricting our discussion to unital $\mathrm{C}^*$-algebras.

\begin{proposition} \label{prop:nuclear}
Every commutative $\mathrm{C}^*$-algebra is nuclear.
\end{proposition}

\begin{proof}
We will show that $A=C(X)$ has the completely positive approximation property. Let $F\subseteq C(X)$ be finite and let $\varepsilon>0$. Choose an open cover $U=(U_1,\dots,U_n)$ of $X$ such that, for every $i\in\{1,\dots,n\}$, $x,y\in U_i$ and $f\in F$, we have $|f(x)-f(y)|<\varepsilon$. Fix a point $x_i\in U_i$ for each $i\in\{1,\dots,n\}$ and let $(\rho_1,\dots,\rho_n)$ be a partition of unity subordinate to $U$. Let $B$ be the finite-dimensional $\mathrm{C}^*$-algebra $\mathbb{C}^n$, and define maps $\psi\colon A\to B$ and $\varphi\colon B\to A$ by
\[
\psi(f) = (f(x_1),\dots,f(x_n)), \quad \varphi(\lambda_1,\dots,\lambda_n) = \sum_{i=1}^n \lambda_i\rho_i.
\]
Then, $\varphi$ and $\psi$ are completely positive (in fact, also contractive and indeed $\psi$ is even a $^*$-homomorphism) and for every $f\in F$ we have
\begin{align*}
\|\varphi\circ\psi(f) - f\| &= \sup_{x\in X}\left|\sum_{i=1}^nf(x_i)\rho_i(x) - \sum_{i=1}^n\rho_i(x)f(x)\right|\\
&\le \sup_{x\in X} \sum_{i=1}^n|f(x_i)-f(x)|\rho_i(x)\\
&< \varepsilon. \qedhere
\end{align*}
\end{proof}

If $\dim X = N < \infty$, then we can say more. Here, $\dim X$ is the \emph{Lebesgue covering dimension} of $X$: it is the least integer $N$ such that every finite open cover $\mathcal{W}$ of $X$ can be refined to a finite open cover $\mathcal{V}$ such that no $N+2$ elements of $\mathcal{V}$ can intersect nontrivially. Equivalently (see \cite[Proposition 1.5]{Kirchberg:2004qy}), every $\mathcal{W}$ admits a finite \emph{$N$-decomposable} refinement, that is, a finite open cover $U=(U_i)_{i\in I}$ such that the index set $I$ can be partitioned into $N+1$ colours $I_0\sqcup\dots\sqcup I_N$ such that $U_{i_1}\cap U_{i_2}=\emptyset$ for every distinct pair of indices $i_1,i_2$ in any $I_k$. In other words, sets of the same colour are pairwise disjoint (though differently coloured sets may intersect).

Now assume that the open cover $U$ that appears in the proof of Proposition~\ref{prop:nuclear} is $N$-decomposable. For $k\in\{0,\dots,N\}$, let $B_k=\mathbb{C}^{|I_k|}$ and $\varphi_k=\varphi|_{B_k}$. Because elements $\rho_i$ of the same colour are orthogonal, we then have the following data associated to a finite set $F\subseteq A$ and $\varepsilon>0$:

\begin{itemize}
\item a completely positive contractive (cpc) map $\psi$ from $A$ to a finite-dimensional $\mathrm{C}^*$-algebra $B$;
\item a decomposition $B=B_0\oplus\dots\oplus B_N$;
\item a completely positive map $\varphi\colon B\to A$ such that each $\varphi_k=\varphi|_{B_k}$ is cpc \emph{order zero} (that is, preserves orthogonality of positive elements) and such that $\|\varphi\circ\psi(f)-f\|<\varepsilon$ for every $f\in F$.
\end{itemize}

By definition \cite[Definition 2.1]{Winter:2010dn}, this means that the \emph{nuclear dimension} $\mathrm{dim}_{\mathrm{nuc}}$ of $A=C(X)$ is $\le N$. (If there is no $N$ that works for an arbitrary $F$ and $\varepsilon$, in particular if $A$ is not nuclear, then $\mathrm{dim}_{\mathrm{nuc}}(A)$ is defined to be $\infty$.)

The argument above demonstrates that $\mathrm{dim}_{\mathrm{nuc}}(C(X)) \le \dim X$ for every compact Hausdorff space $X$. Even more is true: not only is every $\varphi_k$ contractive, but $\varphi=\varphi_0\oplus\dots\oplus\varphi_N$ is as well. By definition \cite[Definition 3.1]{Kirchberg:2004qy}, this means that the \emph{decomposition rank} $\mathrm{dr}$ of $A=C(X)$ is $\le N$. Clearly, $\mathrm{dim}_{\mathrm{nuc}}(A)\le \mathrm{dr}(A)$ for any $\mathrm{C}^*$-algebra $A$. In fact, equality holds for commutative $\mathrm{C}^*$-algebras (see \cite[Proposition 2.4]{Winter:2010dn}):

\[
\mathrm{dim}_{\mathrm{nuc}}(C(X)) = \mathrm{dr}(C(X)) = \dim X.
\]

The \emph{Jiang--Su algebra $\mathcal{Z}$}, constructed in \cite{Jiang:1999hb}, is the unique object in $\mathcal{E}$ with $\mathrm{Ell}(\mathcal{Z})\cong\mathrm{Ell}(\mathbb{C})$. A $\mathrm{C}^*$-algebra $A$ is \emph{$\mathcal{Z}$-stable} if $A\otimes\mathcal{Z}\cong A$. By \cite[Theorem A]{Castillejos:2021wm} (see also \cite{Castillejos:2020wv}), finite nuclear dimension is equivalent to $\mathcal{Z}$-stability for infinite-dimensional, simple, separable, nuclear $\mathrm{C}^*$-algebras. So, in Definition~\ref{def:category}, finite nuclear dimension can be replaced by nuclearity and $\mathcal{Z}$-stability.

Strikingly (see \cite[Corollary C]{Castillejos:2021wm}), the only possible finite values of the nuclear dimension or decomposition rank of a simple, unital $\mathrm{C}^*$-algebra $A$ are $0$ (if and only if $A$ is AF) or $1$. And while the difference in definition between $\mathrm{dr}$ and $\mathrm{dim}_{\mathrm{nuc}}$ might seem small, the difference in effect is dramatic. By \cite[Theorem 5.3]{Kirchberg:2004qy}, every separable $\mathrm{C}^*$-algebra of finite decomposition rank must be quasidiagonal and therefore stably finite. By contrast, as pointed out in \cite[Corollary 2.13]{Jiang:1999hb}, every (unital) purely infinite, simple, separable, nuclear $\mathrm{C}^*$-algebra $A$ is $\mathcal{Z}$-stable, and therefore has finite nuclear dimension.

\section{$K$-theory and the UCT} \label{section:ktheory}

\subsection{$K$-theory} We briefly and somewhat informally recall some of the basics of operator $K$-theory. For details, the reader should consult, for example, \cite{Rordam:2000fk} or \cite{Blackadar:1998qf}.

\begin{definition}
Projections $p$ and $q$ in a $\mathrm{C}^*$-algebra $A$ are \emph{Murray--von Neumann equivalent}, written $p\sim q$, if there exists $v\in A$ such that $v^*v=p$ and $vv^*=q$. (Such a $v$ is a \emph{partial isometry}: $vv^*v=v$.) The \emph{Murray--von Neumann semigroup} $V(A)$ of $A$ is the monoid whose elements are equivalence classes of projections in $M_\infty(A)=\bigcup_{n\in\mathbb{N}}M_n(A)$, with addition $[p]+[q] = \left[\begin{pmatrix}p&0\\0&q\end{pmatrix}\right]$ and zero element $[0]$.
\end{definition}

One can replace Murray--von Neumann equivalence by unitary equivalence or homotopy of projections since these notions yield the same equivalence classes in $M_\infty(A)$.

\begin{example} \label{ex:mvn}
\begin{enumerate}[1.]
\item For any $n\in\mathbb{N}$, $V(M_n) \cong V(\mathbb{C}) \cong V(\mathbb{K}(\mathcal{H})) \cong \mathbb{N}_0=\mathbb{N}\cup\{0\}$ via $[p]\mapsto\mathrm{rank}(p)$.
\item For any compact Hausdorff space $X$, $V(C(X))$ is isomorphic to the semigroup of isomorphism classes of complex vector bundles over $X$. In particular, the classes of the trivial bundle and the Bott bundle over $S^2$ generate $V(C(S^2)) \cong \mathbb{N}_0^2$.
\item $V(C_0(X))=0$ for any connected, locally compact space $X$.
\item Like $C_0(X)$, the $\mathrm{C}^*$-algebra
\[
A=\{f \in C([0,1],M_2) \mid f(0)=\mathrm{diag}(a,a),\: f(1)=\mathrm{diag}(a,0),\: a\in\mathbb{C}\}
\]
is \emph{stably projectionless}: there are no nonzero projections in $M_n(A)$ for any $n\in\mathbb{N}$. For such algebras, we always have $V(A)=0$.
\end{enumerate}
\end{example}

Properties of $V(\cdot)$ include
\begin{itemize}
\item functoriality: every $^*$-homomorphism $\varphi\colon A\to B$ induces a monoid homomorphism $\varphi_*\colon V(A)\to V(B)$;
\item homotopy invariance: if $\varphi,\psi\colon A\to B$ are homotopic $^*$-homomorphisms (that is, there is a $^*$-homomorphism $\Phi\colon A\to C([0,1])\otimes B$ such that $\mathrm{ev}_0\circ\Phi=\varphi$ and $\mathrm{ev}_1\circ\Phi=\psi$), then $\varphi_*=\psi_*$ (so in particular, homotopy equivalent $\mathrm{C}^*$-algebras have isomorphic semigroups);
\item additivity: $V(A\oplus B) \cong V(A)\oplus V(B)$;
\item continuity: $V(\varinjlim(A_i,\varphi_{ij})) \cong \varinjlim (V(A_i),(\varphi_{ij})_*)$.
\end{itemize}

\begin{definition}
If $A$ is unital, then $K_0(A)$ is the Grothendieck group associated to $V(A)$, that is, the abelian group generated by formal differences $[p]-[q]$ of classes in $V(A)$, with the identification
\[
[p_1]-[q_1] = [p_2]-[q_2] \iff [p_1]+[q_2]+[r] = [p_2]+[q_1]+[r] \quad\text{ for some } [r]\in V(A).
\]
If $A$ is nonunital, then $K_0(A)$ is defined relative to the minimal unitisation $\widetilde{A}$ of $A$: it is the kernel of the group homomorphism $\pi_*\colon K_0(\widetilde{A})\to K_0(\mathbb{C})$ induced by the quotient map $\pi\colon \widetilde{A}\to\mathbb{C}$, that is, $\pi_*([p]-[q]) = [\pi(p)]-[\pi(q)]$. In either case, unital or not, we write $K_0(A)_+$ for the image of $V(A)$ in $K_0(A)$.
\end{definition}

The `$[r]$' in the definition is unnecessary precisely when $V(A)$ has cancellation (in which case the map $V(A)\to K_0(A)_+$ is injective). A necessary condition for this is \emph{stable finiteness}.

\begin{definition} \label{def:finite}
A $\mathrm{C}^*$-algebra is:
\begin{itemize}
\item \emph{finite} if for all projections $p,q\in A$, if $p\le q$ and $p\sim q$, then $p=q$;
\item \emph{stably finite} if $M_n(A)$ is finite for every $n\in\mathbb{N}$;
\item \emph{infinite} if it is not finite;
\item (assuming that $A$ is simple) \emph{purely infinite} if $A\ne\mathbb{C}$ and for every pair of nonzero positive elements $a,b\in A$, there exists $x\in A$ such that $b=xax^*$.
\end{itemize}
\end{definition}

\begin{remark} \label{remark:finite}
Here, we collect some facts about (in)finiteness that are particularly relevant to classification.
\begin{enumerate}[1.]
\item Infiniteness of a unital $\mathrm{C}^*$-algebra $A$ is equivalent to the existence of a nonunitary isometry in $A$. This is the case for $\mathbb{B}(\mathcal{H})$, which is in fact \emph{properly infinite}: there exist two isometries with orthogonal ranges. Stable finiteness is in general stronger than finiteness: see \cite{Rordam:1998aa} and also \cite[Corollary 7.2]{Rordam:2003rz}, which describes a \emph{simple}, nuclear, finite $\mathrm{C}^*$-algebra $A$ such that $M_n(A)$ is properly infinite for all sufficiently large $n$.
\item If $A$ is stably finite, then $(K_0(A),K_0(A)_+)$ is an ordered abelian group, meaning that $K_0(A)_+\cap(-K_0(A)_+) = \{0\}$ and $K_0(A)=K_0(A)_+-K_0(A)_+$. On the other hand, if $A$ is unital and properly infinite, then $K_0(A)_+=K_0(A)$.
\item Recall that $A$ is \emph{exact} if
\[
\begin{tikzcd}[column sep=small]0 \arrow[r] & A\otimes J \arrow[r] & A\otimes B \arrow[r] & A\otimes B/J \arrow[r] & 0\end{tikzcd}
\]
is a short exact sequence of $\mathrm{C}^*$-algebras whenever
\[
\begin{tikzcd}[column sep=small]0 \arrow[r] & J \arrow[r] & A \arrow[r] & A/J \arrow[r] & 0\end{tikzcd}
\]
is, where $\otimes$ denotes the minimal tensor product. Every nuclear $\mathrm{C}^*$-algebra is exact, $\mathrm{C}^*_r(\mathbb{F}_2)$ is exact but not nuclear, and $\mathrm{C}^*(\mathbb{F}_2)$ (hence $\mathbb{B}(\mathcal{H})$ for infinite-dimensional $\mathcal{H}$) is not exact (see \cite{Wassermann:1994rz}). If $A$ is simple, unital and exact, then $A$ is stably finite if and only if $T(A)\ne\emptyset$ (see \cite[\S1.1.3]{Rordam:2002yu}, and Section~\ref{section:traces} below for a discussion of the trace space $T(A)$). 
\item A simple, exact, $\mathcal{Z}$-stable $\mathrm{C}^*$-algebra is either stably finite or purely infinite (see \cite[Corollary 5.1]{Rordam:2004kq}). From the previous two remarks, we see that the Elliott invariant for the purely infinite portion of the classifiable category reduces to $(K_0(\cdot),[1],K_1(\cdot))$.
\end{enumerate}
\end{remark}

A sufficient condition for cancellation in $V(A)$ is \emph{stable rank one}, meaning that the invertible elements in $\widetilde{A}$ are dense in $\widetilde{A}$. This holds, for example, for $A=C(X)$ whenever $\dim X\le 1$, for AF algebras, for simple, stably finite, $\mathcal{Z}$-stable $\mathrm{C}^*$-algebras (see \cite{Rordam:2002yu,Fu:2022aa}), and for $\mathrm{C}^*_r(\mathbb{F}_n)$ for any $2\le n\le \infty$ (see \cite{Dykema:1997aa}).

$K_0(\cdot)$ inherits the properties of $V(\cdot)$ (functoriality, homotopy invariance, additivity and continuity), and we can use these properties to compute the $K_0$-groups of many examples.

\begin{example}
\begin{enumerate}[1.]
\item For any $n\in\mathbb{N}$, $K_0(M_n) \cong K_0(\mathbb{C}) \cong K_0(\mathbb{K}(\mathcal{H})) \cong \mathbb{Z}$. In general, $K_0(A) \cong K_0(M_n(A)) \cong K_0(A\otimes\mathbb{K}(\mathcal{H}))$ for any $A$.
\item On the other hand, if $\mathcal{H}$ is infinite dimensional, then $K_0(\mathbb{B}(\mathcal{H}))=0$: for any projection $p$ over $\mathbb{B}(\mathcal{H})$, we have $\begin{pmatrix}p&0\\0&1\end{pmatrix} \sim \begin{pmatrix}0&0\\0&1\end{pmatrix}$ and so $[p]+[1] = [0]+[1]$ in $V(\mathbb{B}(\mathcal{H}))$, which implies that $[p]=[0]$.
\item The $K_0$-group of the CAR algebra $M_{2^\infty}=\bigotimes_{n\in\mathbb{N}}M_2$ can be computed using continuity: $K_0(M_{2^\infty}) \cong \varinjlim(\begin{tikzcd}\mathbb{Z} \arrow[r,"\times2"] & \mathbb{Z} \arrow[r,"\times2"] & \mathbb{Z} \arrow[r,"\times2"] & \dots\end{tikzcd}) \cong \mathbb{Z}\left[\frac{1}{2}\right]$. A similar computation can be done for any UHF algebra. In particular, the $K_0$-group of the universal UHF algebra $\mathcal{Q}$ is $\mathbb{Q}$.
\item The cone $C_0((0,1],A)$ over any $\mathrm{C}^*$-algebra $A$ is null homotopic, so has zero $K$-theory.
\item For any compact Hausdorff space $X$, $K_0(C(X)) \cong K^0(X)$, the topological $K$-theory of $X$. If $X$ is contractible, then $K_0(C(X)) \cong K_0(\mathbb{C}) \cong \mathbb{Z}$.
\end{enumerate}
\end{example}

\begin{definition}[Higher $K$-groups]
For $n\ge1$, $K_n(A)$ is defined to be the $K_0$-group of the $n$th suspension of $A$, that is, $K_n(A)=K_0(A\otimes C_0(\mathbb{R}^n))$.
\end{definition}

It can be shown (see \cite[Theorem 8.2.2]{Blackadar:1998qf}) that $K_1(A)\cong K_1(\widetilde{A})$ is isomorphic to the abelian group generated by homotopy classes in $U_\infty(\widetilde{A}) = \bigcup_{n\in\mathbb{N}}U_n(\widetilde{A})$, where $U_n(\widetilde{A})\hookrightarrow U_{n+1}(\widetilde{A})$ via $u\mapsto\mathrm{diag}(u,1)$. The identity element is $[1]$, the class of any unitary homotopic to the identity in some $M_n(\widetilde{A})$, and addition is (or can be shown to be) given by $[u]+[v]=[uv]$. In particular, $-[u]=[u^*]$, and any element in $K_1(A)$ is represented by a single unitary (rather than a formal difference).

\begin{example}
\begin{enumerate}[1.]
\item By the Borel functional calculus, any unitary in $\mathbb{B}(\mathcal{H})$ (for any $\mathcal{H}$) is homotopic to the identity. So, $K_1(\mathbb{B}(\mathcal{H}))=0$.
\item By additivity, continuity and homotopy invariance, $K_1(A)=0$ for any AF or approximately interval (AI) algebra $A$.
\item If $A$ is either purely infinite or stable rank one, then $K_1(A) \cong U(A)/U(A)_0$, the group of homotopy classes of unitaries in $A$ (rather than in matrix algebras over $A$). In particular, $K_1(C(S^1))\cong\mathbb{Z}$ via the winding number. (Note that $K_0(C(S^1))$ is also $\mathbb{Z}$, since every complex vector bundle over the circle is trivial.)
\end{enumerate}
\end{example}

\sloppy
Any ideal $J$ in a $\mathrm{C}^*$-algebra $A$, corresponding to a short exact sequence $\begin{tikzcd}[column sep=small]0 \arrow[r] & J \arrow[r,"\iota"] & A \arrow[r,"\pi"] & A/J \arrow[r] & 0,\end{tikzcd}$ produces a long exact sequence
\[
\begin{tikzcd}[column sep=small]
\dots \ar[r,"\partial"] & K_1(J) \arrow[r,"\iota_*"] & K_1(A) \arrow[r,"\pi_*"] & K_1(A/J) \arrow[r,"\partial"] & K_0(J) \arrow[r,"\iota_*"] & K_0(A) \arrow[r,"\pi_*"] & K_0(A/J)
\end{tikzcd}
\]
\fussy
(see \cite[Theorem 8.3.5]{Blackadar:1998qf}). Actually, by Bott periodicity (see \cite[Theorem 9.2.1]{Blackadar:1998qf}), there is a natural isomorphism $\beta_A\colon K_0(A)\to K_2(A)$ and the long exact sequence then becomes the \emph{six term exact sequence of $K$-theory}:
\begin{equation} \label{eqn:6term}
\begin{tikzcd}
K_0(J) \arrow[r] & K_0(A) \arrow[r] & K_0(A/J) \arrow[d,"\exp"]\\
K_1(A/J) \arrow[u,"\partial"] & K_1(A) \arrow[l] & K_1(J) \arrow[l]
\end{tikzcd}
\end{equation}
The exponential map $\exp\colon K_0(A/J) \to K_1(J)$ (which represents the obstruction to lifting projections from $A/J$ to $A$) can be computed as
\begin{equation} \label{eqn:exp}
\exp([p_1]-[p_2]) = [e^{2\pi ix_1}] - [e^{2\pi ix_2}],
\end{equation}
where $x_j\in M_n(A)$ is a self-adjoint lift of $p_j\in M_n(A/J)$. The abstract Fredholm index $\partial\colon K_1(A/J)\to K_0(J)$ (which represents the obstruction to lifting unitaries) can be computed as
\begin{equation} \label{eqn:index}
\partial([u]) = [1-v^*v] - [1-vv^*],
\end{equation}
where $u\in M_n(\widetilde{A/J})$ is a unitary and $v\in M_k(\widetilde{A})$ ($k\ge n$) is a partial isometry lift of $u$ if such a lift exists, or otherwise a lift of $\begin{pmatrix}u&0\\0&0\end{pmatrix}$ (see \cite[Proposition 9.2.4]{Rordam:2000fk}). Equipped with these formulae, the six term sequence is a very useful tool for computing $K$-theory.

\begin{example} \label{ex:free}
\sloppy
\begin{enumerate}[1.]
\item \label{ex1} The Toeplitz algebra $\mathcal{T}$ is the universal unital $\mathrm{C}^*$-algebra generated by an isometry (see \cite[Example 9.4.4]{Rordam:2000fk}). Concretely, $\mathcal{T}\cong\mathrm{C}^*(S)$, where $S\in\mathbb{B}(\ell^2\mathbb{N})$ is the unilateral shift. It fits into the short exact sequence
\[
\begin{tikzcd}[column sep=small]0 \arrow[r] & \mathbb{K}(\mathcal{H}) \arrow[r] & \mathcal{T} \arrow[r,"\sigma"] & C(S^1) \arrow[r] &0 \end{tikzcd},
\]
where $\sigma$ is the symbol map that sends $S$ to the identity map $z$. It is a straightforward exercise to use \eqref{eqn:6term} and \eqref{eqn:index} to compute $K_0(\mathcal{T})\cong\mathbb{Z}$ and $K_1(\mathcal{T})=0$.
\item \label{ex2} The Jiang--Su algebra $\mathcal{Z}$ is constructed in \cite{Jiang:1999hb} as a limit of \emph{prime dimension drop algebras}
\[
Z_{p,q} = \{f\in C([0,1],M_p\otimes M_q) \mid f(0) \in M_p\otimes 1_q,\: f(1) \in 1_p\otimes M_q\}
\]
where $p,q\ge1$ are coprime integers. In fact, $\mathcal{Z}$ is the unique inductive limit of building blocks of this form that is simple and monotracial. Applying \eqref{eqn:6term} and \eqref{eqn:exp} to the ideal $J=C_0((0,1),M_p\otimes M_q)$ of $Z_{p,q}$, one can show that $K_0(\mathcal{Z})\cong K_0(Z_{p,q}) \cong \mathbb{Z}$ (generated by the class of the unit) and $K_1(\mathcal{Z}) \cong K_1(Z_{p,q}) = 0$.
\item \label{ex3} The $\mathrm{C}^*$-version of the free group factor problem is easily solved using $K$-theory. It can be shown that, for any finite $n$, $K_0(\mathrm{C}^*_r(\mathbb{F}_n)) \cong K_0(\mathrm{C}^*(\mathbb{F}_n)) \cong \mathbb{Z}$ and $K_1(\mathrm{C}^*_r(\mathbb{F}_n)) \cong K_1(\mathrm{C}^*(\mathbb{F}_n)) \cong \mathbb{Z}^n$ (see \cite{Wassermann:1994rz} for an elementary proof). In particular, these $\mathrm{C}^*$-algebras are all pairwise nonisomorphic.
\end{enumerate}
\fussy
\end{example}

Example~\ref{ex:free}.\ref{ex3} demonstrates the use of $K$-theory in distinguishing isomorphism classes of $\mathrm{C}^*$-algebras. Conversely, its positive role in classification was first made evident in \cite{Elliott:1976kq}. Elliott showed that, if $A$ and $B$ are unital AF-algebras, then $A\cong B$ if and only if $(K_0(A),K_0(A)_+,[1_A])\cong (K_0(B),K_0(B)_+,[1_B])$. In \cite{Farah:2014vt}, it is shown that $\mathcal{E}$ cannot be classified by countable structures. In particular, ordered $K$-theory is not enough. The missing ingredient is the tracial functor $T(\cdot)$, which is the major source of the complexity of the Elliott invariant and will be discussed in the next section. For a unital AF algebra, or more generally, a unital exact $\mathrm{C}^*$-algebra $A$ of real rank zero (which means that every self-adjoint element of $A$ can be approximated by \emph{invertible} self-adjoint elements), $T(A)$ is recoverable as the space of states on $K_0(A)$ (see \cite[Theorem 1.1.11]{Rordam:2002yu}). This explains the absence of tracial data in Elliott's theorem. On the other hand, if one also takes $K_1$ and $T(\cdot)$ into account, then Elliott's \emph{intertwining argument} is very broadly applicable (see, for example, \cite{Elliott:1993ai,Elliott:1993kq,Jiang:1999hb}).

\subsection{The UCT} \label{subsection:uct} Finally, we turn to the universal coefficient theorem (UCT), following the point of view of \cite[Definition 2.4.5]{Rordam:2002yu} (which is based on an observation of Skandalis \cite{Skandalis:1988aa}).

\begin{definition}
A separable $\mathrm{C}^*$-algebra $A$ is in the \emph{UCT class} $\mathcal{N}$ if it is $KK$-equivalent to a commutative $\mathrm{C}^*$-algebra.
\end{definition}

Here, $KK_*(\cdot,\cdot)$ is Kasparov's bivariant $K$-theory, which includes both $K$-theory $KK_*(\mathbb{C},\cdot)$ and $K$-homology $KK_*(\cdot,\mathbb{C})$ as special cases. A meaningful discussion of $KK$-theory is well beyond the scope of this article (see \cite{Blackadar:1998qf,Higson:1990qd,Higson:2000to} instead), but we note the useful point of view that elements of $KK(A,B)$ can be thought of as generalised morphisms from $A$ to $B$. In particular, $^*$-homomorphisms define $KK$-classes. A key feature of $KK$-theory is the (bilinear, associative) \emph{Kasparov product} $KK(A,B) \times KK(B,C) \to KK(A,C)$ that extends composition of $^*$-homomorphisms and provides $KK(A,A)$ with the structure of a ring with unit $\iota_A=KK(\mathrm{id}_A)$. $KK$-equivalence between $A$ and $B$ means the existence of $x\in KK(A,B)$ and $y\in KK(B,A)$ with $x\times y =\iota_A$ and $y\times x =\iota_B$. If $A$ and $B$ are $KK$-equivalent (for example, if they are isomorphic as $\mathrm{C}^*$-algebras), then there are induced group isomorphisms $K_*(A) \cong K_*(B)$. The converse holds under the UCT.

\begin{theorem}[UCT]
A separable $\mathrm{C}^*$-algebra is in $\mathcal{N}$ if and only if, for every separable $\mathrm{C}^*$-algebra $B$, there is a short exact sequence of $\mathbb{Z}/2\mathbb{Z}$-graded abelian groups
\[
\begin{tikzcd}[column sep=small]
0 \arrow[r] & \mathrm{Ext}^1_\mathbb{Z}(K_*(A),K_{*+1}(B)) \arrow[r] & KK_*(A,B) \arrow[r] & \mathrm{Hom}(K_*(A),K_*(B)) \arrow[r] & 0
\end{tikzcd}
\]
that splits unnaturally.
\end{theorem}

In particular:
\begin{itemize}
\item $KK_*(A,B) \cong \mathrm{Hom}(K_*(A),K_*(B))$ if $K_*(A)$ is free or $K_*(B)$ is divisible;
\item if $A,B\in\mathcal{N}$, then $A$ is $KK$-equivalent to $B$ if and only if $K_*(A)\cong K_*(B)$ (see \cite[\S23.10]{Blackadar:1998qf}).
\end{itemize}

By work of Skandalis \cite{Skandalis:1988aa} (a `sobering example' in Higson's words \cite{Higson:1990qd}), it is known that not every separable $\mathrm{C}^*$-algebra is in $\mathcal{N}$. On the other hand, one of the major open problems in operator algebras is to decide whether every separable \emph{nuclear} $\mathrm{C}^*$-algebra is in $\mathcal{N}$. We write $\mathcal{N}_{\mathrm{nuc}}$ for the class of nuclear $\mathrm{C}^*$-algebras in $\mathcal{N}$, and refer to it as the \emph{bootstrap class} of Rosenberg and Schochet \cite{Rosenberg:1987fp}. The bootstrap class $\mathcal{N}_{\mathrm{nuc}}$ contains $\mathbb{C}$, is closed under countable inductive limits and $KK$-equivalence, and if two $\mathrm{C}^*$-algebras in a short exact sequence are in $\mathcal{N}_{\mathrm{nuc}}$, then so is the third. From this, one can show that $\mathcal{N}_{\mathrm{nuc}}$ contains any countable inductive limit of type $\rm{I}$ $\mathrm{C}^*$-algebras, and is closed under tensor products and crossed products by actions of $\mathbb{Z}$ or $\mathbb{R}$ (see \cite[22.3.5]{Blackadar:1998qf}). By a theorem of Tu \cite{Tu:1999aa}, every amenable groupoid $\mathrm{C}^*$-algebra also lives in $\mathcal{N}_{\mathrm{nuc}}$.

A common use of the UCT in classification is  Schochet's K\"{u}nneth theorem for tensor products \cite{Schochet:1982vp}: if $A\in\mathcal{N}_{\mathrm{nuc}}$, then for any separable $\mathrm{C}^*$-algebra $B$, there is a short exact sequence
\[
\begin{tikzcd}[column sep=small]
0 \arrow[r] & K_*(A)\otimes K_*(B) \arrow[r,"\alpha"] & K_*(A\otimes B) \arrow[r] & \mathrm{Tor}^1_\mathbb{Z}(K_*(A),K_{*+1}(B)) \arrow[r] & 0
\end{tikzcd}
\]
that splits unnaturally. Note that $\alpha$ is an isomorphism if $K_*(A)$ or $K_*(B)$ is torsion free. In particular, $K_*(A\otimes\mathcal{Z})\cong K_*(A)$ for any separable $\mathrm{C}^*$-algebra $A$.

Beyond this, the role of the UCT in classification as a tool for solving lifting problems is made most clear in the work of Schafhauser \cite{Schafhauser:2020vs} and his coauthors \cite{Carrion:wz}. It should however be emphasised that the Kirchberg--Phillips theorem (that is, the classification of purely infinite, simple, separable, nuclear $\mathrm{C}^*$-algebras) is $KK$-theoretic, with the $K$-theoretic classification then being a consequence of the UCT. See \cite{Kirchberg:2000kq}, \cite[\S4.2]{Phillips:2000fj} and \cite[Chapter 8]{Rordam:2002yu}.

\section{Traces and the Cuntz semigroup} \label{section:traces}

\subsection{Traces} In these notes, a \emph{trace} on a unital $\mathrm{C}^*$-algebra $A$ means a tracial state, that is, a positive linear functional $\tau\colon A\to\mathbb{C}$ with $\|\tau\|=1=\tau(1)$ that satisfies the trace identity $\tau(uau^*)=\tau(a)$ for every $a\in A$ and unitary $u\in A$. (This is equivalent to requiring that $\tau(xy)=\tau(yx)$ for every $x,y\in A$; see \cite[Proposition 5.2.2]{Pedersen:1979rt}.) We write $T(A)$ for the set of traces on $A$ and refer to it as the \emph{trace space of $A$}.

The trace space is a compact, convex subset of the unit ball of $A^*$, equipped with the $w^*$-topology (that is, $\tau_i\to\tau$ if and only if $\tau_i(a)\to\tau(a)$ for every $a\in A$). By a result of Thoma \cite{Thoma:1964aa} (see also \cite[Theorem 3.1]{Pedersen:1969kq}), $T(A)$ is in fact a \emph{metrisable Choquet simplex}. To motivate the definition of such a structure, we start from the understanding that, for an integer $n\ge0$, an \emph{$n$-dimensional simplex} $K$ is the convex hull of $n+1$ affinely independent points $\{e_0,\dots,e_n\}$ in a vector space. These points form the \emph{extreme boundary} $\partial_e(K)$ of $K$, which is defined by the property that $e\in\partial_e(K)$ if and only if the only way to write $e$ as a convex combination $e=\lambda x + (1-\lambda)y$ of points $x,y\in K$ with $\lambda\in(0,1)$ is with $x=y=e$.

Let $K=\mathrm{conv}\{e_0,\dots,e_n\}$ as above and fix $x\in K$ with barycentric coordinates $(\lambda_0,\dots,\lambda_n)$, that is, $x=\sum_{i=0}^n\lambda_i e_i$. Write $\mathrm{Aff}(K)$ for the set of continuous affine maps $f\to\mathbb{R}$ (notation that we will continue to use for any compact convex set $K$). Then, $f(x) = \sum_{i=0}^n\lambda_i f(e_i)$ for every $f\in\mathrm{Aff}(K)$, which can be rewritten as
\begin{equation} \label{eqn:rep}
f(x) = \int_K f\,d\mu_x \quad \text{ for every } f\in\mathrm{Aff}(K).
\end{equation}
Here, $\mu_x=\sum_{i=0}^n\lambda_i\delta_{e_i}$, $\delta_y$ denoting the point mass at $y$. We say that $\mu_x$ is a \emph{representing measure for $x$} whenever \eqref{eqn:rep} holds. While $\delta_x$ would, trivially, be another representing measure for $x$, affine independence of the extreme points of $K$ implies that $\mu_x=\sum_{i=0}^n\lambda_i\delta_{e_i}$ is the unique representing measure \emph{that is supported on $\partial_e(K)$}.

\begin{definition} \label{def:choquet}
A compact, convex subset $K$ of a Hausdorff topological vector space is said to be a \emph{metrisable Choquet simplex} if $K$ is metrisable and every $x\in K$ has a unique representing Borel probability measure $\mu_x$ supported on $\partial_e(K)$ (in the sense that $\mu_x(K\setminus\partial_e(K))=0$).
\end{definition}

It is a straightforward exercise to show that, in the metrisable setting, $\partial_e(K)$ is a $G_\delta$ (in particular, Borel) subset of $K$. (This need not be true in general, and nonseparable integral representation theory is somewhat more delicate. See \cite[Chapter 4]{Phelps:2001rz}.) Moreover, by Choquet's theorem (see \cite[Chapter 3]{Phelps:2001rz}), a representing measure on $\partial_e(K)$ always exists. It is the requirement of uniqueness that characterises simplices. An equivalent characterisation is that the vector space generated by the cone over $K$ is a vector lattice (see \cite[Chapter 10]{Phelps:2001rz}), which is in fact the property of $T(A)$ that is demonstrated in \cite{Pedersen:1969kq} and \cite{Thoma:1964aa}.

\begin{example}
\begin{enumerate}[1.]
\item By design, finite-dimensional simplices are metrisable Choquet simplices. Note that an $n$-simplex is affinely homeomorphic to the trace space of the direct sum of $n+1$ matrix algebras.
\item The space $K=\mathcal{M}_1^+(X)$ of Borel probability measures on a compact metrisable space $X$ (which by the Riesz representation theorem is affinely homeomorphic to $T(C(X)$) is a particular kind of Choquet simplex called a \emph{Bauer simplex}, meaning that $\partial_e(K)$ is compact (as in this case it is homeomorphic to $X$ via $x\mapsto\delta_x$). In fact, every Bauer simplex is of this form (see \cite[Corollary \rm{II}.4.2]{Alfsen:1971hl}). For Bauer simplices, Choquet's theorem gives the same information as Krein--Milman (which in general provides a representing measure supported on $\overline{\partial_e(K)}$).
\item The \emph{Poulsen simplex} is an infinite-dimensional metrisable Choquet simplex with the defining property that $\partial_e(K)$ is dense in $K$. It is the `Fra\"{i}ss\'{e} limit' of the class of finite-dimensional simplices (see \cite{Conley:2018vn}).
\end{enumerate}
\end{example}

The Poulsen simplex, or indeed \emph{any} simplex, is a possible trace space for objects in $\mathcal{E}$. This is implicit in the fact mentioned in Section~\ref{section:category} that all possibilities for the Elliott invariant are attained. Actually, if one is only interested in constructing a classifiable $\mathrm{C}^*$-algebra with a prescribed simplex $\Delta$ as its trace space, this goal can be achieved in the category of AF algebras \cite{Goodearl:1977dq,Blackadar:1980zr} or AI algebras \cite{Thomsen:1994qy} using the Lazar--Lindenstrauss theorem \cite[Theorem 5.2 and its Corollary]{Lazar:1971kx}, which shows that $\Delta$ is the projective limit of a sequence of finite-dimensional simplices with connecting maps that are continuous, affine and surjective. One could almost take this as the \emph{definition} of a metrisable Choquet simplex, a point of view to which we will return in Section~\ref{section:random}. On the other hand, in Section~\ref{section:chaos} we will stick with Definition~\ref{def:choquet}.

\subsection{The pairing map} Every trace $\tau\in T(A)$ extends to a (non-normalised) trace $\mathrm{Tr}_n\otimes\tau$ on any matrix algebra $M_n(A)\cong M_n\otimes A$ over $A$ (indeed, to a densely finite lower semicontinuous trace on $A\otimes\mathbb{K}(\mathcal{H})$) and provides a well-defined map $\tau_*$ on Murray--von Neumann equivalence classes of projections, and therefore on $K_0(A)$. The map $\tau_*\colon K_0(A) \to \mathbb{R}$ is a \emph{state}: it is a group homomorphism that maps $K_0(A)_+$ to $\mathbb{R}_+$ and $[1_A]$ to $1$. If $A$ is exact, then \emph{every} state on $K_0(A)$ is of this form (see \cite[Theorem 1.1.11]{Rordam:2002yu}). Dually, the pairing between $T(A)$ and $K_0(A)$ is captured by the group homomorphism $\rho_A\colon K_0(A) \to \mathrm{Aff}(T(A))$, $\rho_A(x)(\tau) = \tau_*(x)$. For the stably finite, real rank zero objects in $\mathcal{E}$, the image of $\rho_A$ is uniformly dense in $\mathrm{Aff}(T(A))$ (see \cite[Theorem 6.9.3]{Blackadar:1998qf}).

The necessity of $T(\cdot)$ as a component of the Elliott invariant is demonstrated by \cite{Goodearl:1992hl}. For an example showing that the pairing is also needed, see \cite[p.29]{Rordam:2002yu}. Moreover, for the stably finite objects in $\mathcal{E}$, the pairing can be used to recover the order structure on $K_0(A)$ (see  \cite[Theorem 6.8.5]{Blackadar:1998qf} and \cite[Theorem 1]{Gong:2000kq}, which shows that weak unperforation, meaning $x>0\iff nx>0$ for some $n\ge1$, holds in $K_0(A)$ for any simple, unital, $\mathcal{Z}$-stable $\mathrm{C}^*$-algebra $A$). Because of this, the positive cone $K_0(A)_+$  is suppressed in some modern presentations of classification \cite{Carrion:wz}.

\subsection{The Cuntz semigroup}

Elliott's classification of AF algebras does not require simplicity, and in fact nonsimple \emph{AI} algebras are also classifiable. However, the invariant in this case is not $\mathrm{Ell}$ but rather $\mathrm{Cu}$ (which in general encodes the ideal lattice of a $\mathrm{C}^*$-algebra; see \cite[Theorem 5.5]{Gardella:2022aa}).

\begin{definition}
Let $A$ be a $\mathrm{C}^*$-algebra. For $a,b\in A_+$, say that $a$ is \emph{Cuntz subequivalent to $b$}, written $a\lesssim b$, if there is a sequence $(x_n)_{n\in\mathbb{N}}$ in $A$ such that $x_nbx_n^*\to a$ in norm. If $a\lesssim b$ and $b\lesssim a$, say that $a$ is \emph{Cuntz equivalent to $b$} and write $a\sim b$. The \emph{Cuntz semigroup} $\mathrm{Cu}(A)=(A\otimes\mathbb{K}(\mathcal{H}))_+/\sim$ is a positively ordered monoid with the addition $[a]+[b]=\left[\begin{pmatrix}a&0\\0&b\end{pmatrix}\right]$ and order $[a]\le[b]$ if and only if $a\lesssim b$.
\end{definition}

$\mathrm{Cu}(\cdot)$ is a stable, additive, continuous functor from the category of $\mathrm{C}^*$-algebras to the `Cuntz category' (see \cite{Ara:2009cs,Gardella:2022aa}). One of the key features  of objects in this category (which we will not define here, instead referring to \cite[\S4]{Gardella:2022aa}) is the existence of suprema of increasing sequences. This is in contrast to the original `nonstable' version $W(\cdot)$ of $\mathrm{Cu}(\cdot)$, introduced in \cite{Cuntz:1978rz}. Morphisms in the Cuntz category (called \emph{$\mathrm{Cu}$-morphisms}) are monoid morphisms that preserve the order structure, suprema of increasing sequences and also the `compact containment' relation glossed over in the present exposition.

\begin{example} \label{ex:cuntz}
\begin{enumerate}[1.]
\item \label{exit1} If $p,q\in A$ are projections, then $p\lesssim q$ if and only if $p$ is Murray--von Neumann equivalent to a subprojection of $q$. There is therefore a natural map $V(A)\to \mathrm{Cu}(A)$, which is injective if $A$ is stably finite. This need not be true outside of the stably finite setting (indeed, see Example~\ref{ex:cuntz}.\ref{exit5} below).
\item \label{exit2} For any $n\in\mathbb{N}$, $\mathrm{Cu}(M_n) \cong \mathrm{Cu}(\mathbb{C}) \cong \mathrm{Cu}(\mathbb{K}(\mathcal{H})) \cong \overline{\mathbb{N}}=\mathbb{N}\cup\{0,\infty\}$ via $[a]\mapsto\mathrm{rank}(a)$.
\item \label{exit3} In $C(X)$, $f\lesssim g$ if and only if the open support of $f$ is contained in the open support of $g$. So, two continuous functions $X\to[0,\infty)$ are Cuntz equivalent if and only if their open supports are the same.
\item \label{exit4} $\mathrm{Cu}(C([0,1]))\cong\mathrm{Lsc}([0,1],\overline{\mathbb{N}})$. More generally, recall from \cite{Eilers:1998yu} that a \emph{one-dimensional noncommutative CW (NCCW) complex} is a pullback $B$ of the form
\begin{equation} \label{eqn:nccw}
\begin{tikzcd}[column sep=large]
B \arrow[d] \arrow[r] & E \arrow[d,"\varphi=\varphi_0\oplus\varphi_1"]\\
C([0,1],F) \arrow[r,"\rho=\mathrm{ev}_0\oplus\mathrm{ev}_1"] & F\oplus F
\end{tikzcd}
\end{equation}
for some finite-dimensional $\mathrm{C}^*$-algebras $E=\bigoplus_{i=1}^n M_{k_i}$ and $F=\bigoplus_{i=1}^m M_{l_i}$ and $^*$-homomorphisms $\varphi_0,\varphi_1\colon E\to F$ (the gluing maps at the boundary of the interval). By \cite[Corollary 3.5]{Antoine:2011zl}, $\mathrm{Cu}(B)$ is isomorphic to the pullback
\[
\mathrm{Cu}(B) \cong \left\{(f,v) \in \mathrm{Lsc}\left([0,1],\overline{\mathbb{N}}^m\right)\oplus\overline{\mathbb{N}}^n \mid f(0) = M_0v,\: f(1)=M_1v\right\},
\]
where $M_i\in M_{m\times n}(\mathbb{N})$ is the matrix of $\mathrm{Cu}(\varphi_i)$ (or equivalently, the matrix of $K_0(\varphi_i)$).
\item \label{exit5} If $A$ is simple and purely infinite, then so is $A\otimes\mathbb{K}(\mathcal{H})$ (see \cite[Proposition 4.1.8]{Rordam:2002yu}) and by definition (see Section~\ref{section:ktheory}) all nonzero positive elements in $A\otimes\mathbb{K}(\mathcal{H})$ are Cuntz equivalent. It follows that $\mathrm{Cu}(A)\cong\{0,\infty\}$.
\end{enumerate}
\end{example}

Building on \cite{Thomsen:1992kq}, it was shown in \cite{Ciuperca:2008rz} that $\mathrm{Cu}$ classifies unital $^*$-homomorphisms from $C([0,1])$ to $\mathrm{C}^*$-algebras of stable rank one (in the sense described in Theorem~\ref{thm:rob} below). The range of codomains for which this is true was expanded in \cite{Robert:2010rz}, and the range of domains was expanded in \cite{Ciuperca:2011wd} to include functions on trees, and then via a beautiful reduction argument in \cite{Robert:2010qy} to include one-dimensional noncommutative CW (NCCW) complexes with trivial $K_1$.

\begin{theorem} \cite{Robert:2010qy} \label{thm:rob}
$\mathrm{Cu}$ classifies unital $^*$-homomorphisms from (inductive limits of direct sums of $\mathrm{C}^*$-algebras stably isomorphic to) unital, $K_1$-trivial one-dimensional NCCW complexes $B$ to $\mathrm{C}^*$-algebras $A$ of stable rank one, meaning that every $\mathrm{Cu}$-morphism $\alpha$ from $\mathrm{Cu}(B)$ to $\mathrm{Cu}(A)$ with $\alpha([1])= [1]$ can be lifted to a unital $^*$-homomorphism $B\to A$, uniquely up to approximate unitary equivalence.
\end{theorem}

We will make use of this powerful result of Robert in Section~\ref{section:weyl}.

\begin{remark}
\begin{enumerate}[1.]
\item Robert's theorem also holds in the nonunital setting, provided that $\mathrm{Cu}(A)$ is replaced by the augmented invariant $\mathrm{Cu}\,\widetilde{}\,(A)$ (see \cite{Robert:2010qy}).
\item Triviality of $K_1(B)$ is equivalent to surjectivity of $K_0(\varphi_0)-K_0(\varphi_1)\colon \mathbb{Z}^n \to \mathbb{Z}^m$ (where $\varphi_0,\varphi_1$ are as in \eqref{eqn:nccw}). Examples of domains covered by Robert's theorem are $C([0,1],M_n)$, prime dimension drop algebras $Z_{p,q}$ and the stably projectionless algebras described in Example~\ref{ex:mvn} (but not, for example, $C(S^1,M_n)$).
\end{enumerate}
\end{remark}

Every trace $\tau\in T(A)$ yields a functional on $\mathrm{Cu}(A)$, that is, a monoid morphism $\mathrm{Cu}(A)\to[0,\infty]$ that preserves order and suprema of increasing sequences, defined by $d_\tau([a])=\sup_{n\in\mathbb{N}}\tau(a^\frac{1}{n})$ and called a \emph{dimension function}. (In fact, \emph{every} functional on $\mathrm{Cu}(A)$ arises (quasi)tracially; see \cite[Proposition 4.2]{Elliott:2009kq}.)

\begin{example}
\begin{enumerate}[1.]
\item If $A=C(X)$, $\tau\in T(A)$ is the trace corresponding to $\mu\in\mathcal{M}_1^+(X)$ and $f\in C(X)_+$, then $d_\tau(f)$ is the $\mu$-measure of the open support of $f$.
\item If $p\in A$ is a projection, then $d_\tau([p]) = \tau(p)$ for every $\tau\in T(A)$.
\end{enumerate}
\end{example}

Dually, every $[a]\in\mathrm{Cu}(A)$ defines a lower semicontinuous function $\widehat{[a]}\colon T(A) \to [0,\infty]$, $\widehat{[a]}(\tau) = d_\tau(a)$. For the stably finite objects in $\mathcal{E}$, every such functional on $T(A)$ is attainable as $\widehat{[a]}$ for some $a$ (this uses `weak divisibility') and for elements not equivalent to projections, $\widehat{[a]}$ determines the Cuntz class of $a$ (this uses `strict comparison'). In fact, the following holds.
  
\begin{theorem}\cite{Brown:2008mz,Brown:2007rz} \label{thm:bpt}
If $A$ is a simple, separable, unital, exact, $\mathcal{Z}$-stable $\mathrm{C}^*$-algebra with $T(A)\ne\emptyset$, then
\begin{equation} \label{eqn:bpt}
\mathrm{Cu}(A) \cong V(A) \sqcup \mathrm{LAff}_+(T(A)),
\end{equation}
where $\mathrm{LAff}_+(T(A))$ denotes the set of functions $T(A)\to(0,\infty]$ that are pointwise suprema of increasing sequences of strictly positive elements of $\mathrm{Aff}(T(A))$.
\end{theorem}

In particular, Robert's theorem reduces to classification by $\mathrm{Ell}$ for simple inductive limits of $K_1$-trivial one-dimensional NCCW complexes (see \cite[\S6]{Robert:2010qy}).

For the proof of Theorem~\ref{thm:bpt} (and indeed the details of the monoid and order structure on the right-hand side of \eqref{eqn:bpt}), see \cite{Brown:2008mz,Brown:2007rz}, \cite{Ara:2009cs} and also \cite{Elliott:2009kq}, which covers $\mathrm{C}^*$-algebras that are not necessarily unital. As a special case, note that
\[
\mathrm{Cu}(\mathcal{Z}) \cong \mathbb{N}_0 \sqcup (0,\infty],
\]
where $\mathbb{N}_0=\mathbb{N}\cup\{0\}$ (notation that we will continue to use in the sequel).

\section{Application A: the Weyl problem} \label{section:weyl}

By the spectral theorem, normal matrices $a,b\in M_n$ are unitarily equivalent if and only if their (multi)sets of eigenvalues $\sigma(a)=\{\alpha_1,\dots,\alpha_n\}$ and $\sigma(b)=\{\beta_1,\dots,\beta_n\}$ agree. If $a$ and $b$ are self adjoint, then a much finer statement holds: not only does the spectrum of $a$ determine its unitary equivalence class, but a suitable comparison between $\sigma(a)$ and $\sigma(b)$ measures the distance between the unitary orbits of $a$ and $b$. The following is contained in \cite{Weyl:1912aa}.

\begin{theorem}[Weyl]
If $a,b\in A=M_n$ are self adjoint, then
\begin{equation} \label{eqn:weyl}
d_U(a,b)=\delta(a,b),
\end{equation}
where $d_U$ is the \emph{unitary distance}
\begin{equation} \label{eqn:du}
d_U(a,b) = \inf_{u\in U(A)} \|a-ubu^*\|
\end{equation}
and $\delta$ is the \emph{optimal matching distance}
\begin{equation} \label{eqn:om}
\delta(a,b) = \inf_{\sigma\in S_n} \max_{1\le i \le n} |\alpha_i-\beta_{\sigma(i)}|.
\end{equation}
Moreover, the infimum in the definition of $\delta(a,b)$ is attained by listing the eigenvalues of $a$ and $b$ in increasing order.
\end{theorem}

The \emph{Weyl problem} for matrices is to determine whether \eqref{eqn:weyl} holds for a given class of normal matrices. `Class'  as understood in the present article means of a given spectral type. The problem has a positive solution for unitary matrices \cite{Bhatia:1984aa} (that is, spectrum contained in the circle) but not for arbitrary normal matrices \cite{Holbrook:1992aa} (although $d_U$ and $\delta$ are Lipschitz equivalent uniformly over $n$, that is, with Lipschitz constants not depending on the matrix size). The \emph{Weyl problem for $\mathrm{C}^*$-algebras} is to determine classes of $\mathrm{C}^*$-algebras $A$ and normal elements in $A$ for which some version of \eqref{eqn:weyl} holds. Here, we will only consider $\mathrm{C}^*$-algebras that are stably finite (but note that there are results in the context of von Neumann algebras \cite{Hiai:1989aa} and purely infinite $\mathrm{C}^*$-algebras \cite{Skoufranis:2013aa}).

For the rest of this section, $A$ will be a stably finite object in $\mathcal{E}$ (although many of the results hold in greater generality) and $(X,d)$ will be a compact, path-connected metric space. If $X\subseteq\mathbb{C}$, then it potentially represents the shared spectrum of normal elements $a,b\in A$. But we will not restrict to planar spaces and will eventually consider the Weyl problem for unital $^*$-monomorphisms $\varphi,\psi\colon C(X) \to A$. In either case, connectedness affords us a reasonable expectation of a purely measure theoretic computation of the relevant unitary distance, that is, it allows us to avoid $K$-theoretic obstructions associated to projections in $C(X)$.

To start with, we must suitably reinterpret the optimal matching distance $\delta$. The right notion already reveals itself in $M_n$. The normalised matrix trace $\tau$ pulls back via the Gelfand map $C(\sigma(a))\to\mathrm{C}^*(a)$ to a trace on $C(\sigma(a))$, which corresponds to a Borel probability measure $\mu_a$. This measure is of course just the normalised counting measure associated to the multiset of eigenvalues of $a$. By Hall's marriage theorem, it follows that
\begin{equation} \label{eqn:hall}
\delta(a,b) = \inf\{r>0 \mid (\forall \text{ open } U)\: \mu_a(U) \le \mu_b(U_r),\: \mu_b(U) \le \mu_a(U_r)\},
\end{equation}
where $U_r$ is the $r$-neighbourhood of $U$, that is, $\{x\in\mathbb{C} \mid d(x,U) < r\}$. The right-hand side of \eqref{eqn:hall} is known as the \emph{$\infty$-Wasserstein distance} $W_\infty(\mu_a,\mu_b)$ between the measures $\mu_a$ and $\mu_b$ (see the discussion and references in \cite[\S2]{Jacelon:2021vc}) and is exactly how we make sense of the optimal matching distance in tracial $\mathrm{C}^*$-algebras.

If $X\subseteq\mathbb{C}$ (with the Euclidean metric) is the spectrum of a normal element $a\in A$, $\varphi_a\colon C(X)\to \mathrm{C}^*(a)$ is the Gelfand map and $\tau\in T(A)$, we write $\mu_{\tau,a}=\mu_{\varphi_a^*\tau}$, the Borel probability measure corresponding to the trace $\varphi_a^*\tau=\tau\circ\varphi_a\in T(C(X)) \cong \mathcal{M}_1^+(X)$. If $b\in A$ is another normal element with spectrum $X$, we define

\begin{equation} \label{eqn:winf}
W_\infty(a,b) = \sup_{\tau\in T(A)} W_\infty(\mu_{\tau,a},\mu_{\tau,b}).
\end{equation}

The Weyl problem in this setting is: for which $A$ and $X$ can we deduce that $d_U(a,b)=W_\infty(a,b)$?

\begin{theorem} \label{thm:jstv}
With $a,b\in A$ as above, $d_U(a,b)=W_\infty(a,b)$ holds if:
\begin{enumerate}[(i)]
\item \label{jstv1} \cite{Jacelon:2014aa} $X$ is an interval;
\item \label{jstv2} \cite{Jacelon:2021wa,Jacelon:2021vc} $X$ is locally connected (so is a `Peano continuum', that is, a continuous image of $[0,1]$), $A$ is real rank zero with $\partial_e(T(A))$ compact and of finite Lebesgue covering dimension, and $K_1(\varphi_a)=K_1(\varphi_b)=0$.
\end{enumerate}
\end{theorem}

\begin{proof}[Proof of \eqref{jstv1}]
Here, we provide a proof of \cite[Theorem 4.1]{Jacelon:2014aa} that is suitable for unital, classifiable $\mathrm{C}^*$-algebras. By \cite[Theorem 4.5]{Jiang:1999hb}, there exists a simple inductive limit $B$ of prime dimension drop algebras with $T(B)\cong T(A)$. By Theorem~\ref{thm:bpt}, we have $\mathrm{Cu}(A) \cong V(A) \sqcup \mathrm{LAff}_+(T(A))$ and $\mathrm{Cu}(B) \cong \mathbb{N}_0 \sqcup \mathrm{LAff}_+(T(A))$ (with $\mathbb{N}_0$ generated by $[1_B]$). The isomorphism of trace spaces therefore induces a $\mathrm{Cu}$-isomorphism $\beta$ from $\mathrm{Cu}(B)$ to the subsemigroup $S$ of $\mathrm{Cu}(A)$ generated by $[1_A]\in\mathbb{N}$ and $\mathrm{LAff}_+(T(A))$ (the latter of which represents the classes of positive elements in $A\otimes\mathbb{K}(\mathcal{H})$ that are not equivalent to projections). By Theorem~\ref{thm:rob}, this $\mathrm{Cu}$-morphism lifts to an embedding $B\hookrightarrow A$.

The images of the $\mathrm{Cu}$-morphisms $\alpha_a,\alpha_b\colon \mathrm{Cu}(C([0,1])) \to \mathrm{Cu}(A)$ induced by $a$ and $b$ are also contained in $S$ (because every projection over $C([0,1])$ is equivalent to a multiple of the unit). By the existence part of Theorem~\ref{thm:rob}, there are positive elements $a',b'\in B$ such that the diagram
\[
\begin{tikzcd}[column sep=large]
& V(A) \sqcup \mathrm{LAff}_+(T(A)) \cong \mathrm{Cu}(A)\\
\mathrm{Cu}(C([0,1])) \arrow[ur,"\alpha_a{,}\alpha_b"] \arrow[r,dashed,"\alpha_{a'}{,}\alpha_{b'}"] & \mathbb{N}_0 \sqcup \mathrm{LAff}_+(T(A)) \cong \mathrm{Cu}(B) \arrow[u,"\beta"]
\end{tikzcd}
\]
commutes. In other words, we have lifted the $\mathrm{Cu}$-morphisms $\beta^{-1}\circ\alpha_a,\beta^{-1}\circ\alpha_b\colon\mathrm{Cu}(C([0,1])) \to \mathrm{Cu}(B)$. By the uniqueness part of Theorem~\ref{thm:rob}, we have $d_U(a,a') = d_U(b,b') = 0$. Replacing $(a,b)$ by $(a',b')$, the problem is then reduced to showing that $d_U=W_\infty$ inside a dimension drop algebra. This in turn is shown by diagonalising and applying Weyl's theorem pointwise.
\end{proof}

To illustrate the ideas behind the proof of Theorem~\ref{thm:jstv}\eqref{jstv2} (which is the generalisation of \cite[Theorem 4.10]{Jacelon:2021wa} discussed after \cite[Definition 2.6]{Jacelon:2021vc}), let us consider the more abstract setting where $X$ is divorced from the plane. Since the function $\mathrm{id}_X\colon X\to X\subseteq \mathbb{C}$ (corresponding to a normal element with spectrum $X$) is no longer available as an element of $C(X)$, we measure distance relative to
\begin{equation} \label{eqn:lipdef}
\mathrm{Lip}^1(X,d) = \{f\colon X\to \mathbb{R} \mid (\forall x,y\in X)\: |f(x)-f(y)|\le d(x,y)\}.
\end{equation}
For unital $^*$-monomorphisms $\varphi,\psi\colon C(X)\to A$, we define the unitary distance
\begin{equation} \label{eqn:dulip}
d_U(\varphi,\psi) = \inf_{u\in U(A)} \sup_{f\in\mathrm{Lip}^1(X,d)}\|\varphi(f)-u\psi(f)u^*\|
\end{equation}
and the tracial distance
\begin{equation} \label{eqn:omlip}
W_\infty(\varphi,\psi) = \sup_{\tau\in T(A)} W_\infty(\mu_{\varphi^*\tau},\mu_{\psi^*\tau}).
\end{equation}

Note that $\mathrm{Lip}^1(X,d)=\mathbb{R}+\mathrm{Lip}^1(X,d)\cap B_{\mathrm{diam}(X)}(C(X))$ and that (by the Arzel\`{a}--Ascoli theorem) $\mathrm{Lip}^1(X,d)\cap B_{\mathrm{diam}(X)}(C(X))$ is compact. So,
\[
d_U(\varphi,\psi) = \inf_{u\in U(A)} \sup_{f\in\mathrm{Lip}^1(X,d)\cap B_{\mathrm{diam}(X)}(C(X))}\|\varphi(f)-u\psi(f)u^*\|
\]
since $\varphi$ and $\psi$ are assumed to be unital. In particular, $d_U(\varphi,\psi)=0$ if and only if $\varphi$ and $\psi$ are approximately unitarily equivalent.

Now the question is: when do we have $d_U(\varphi,\psi)=W_\infty(\varphi,\psi)$?

\begin{definition} [\cite{Jacelon:2021vc}, Definition 2.6] \label{def:ct}
Say that $(X,d)$ has \emph{continuous transport constant} $\le k_X$ if, for every faithful and diffuse $\mu,\nu\in\mathcal{M}_1^+(X)$ and every $\varepsilon>0$, there is a homeomorphism $h\colon X\to X$ homotopic to $\mathrm{id}_X$ such that $W_\infty(\mu,h_*\nu)<\varepsilon$ (where $h_*\nu$ is the pushforward measure $\nu\circ h^{-1}$) and
\[
d(h,\mathrm{id}_X)=\sup_{x\in X}d(h(x),x) < k_XW_\infty(\mu,\nu) + \varepsilon.
\]
\end{definition}

In other words, the \emph{continuous transport map} $h$ moves the mass distributed by $\nu$ to the $\mu$ distribution as efficiently as allowed by the geometry of $(X,d)$. Note that we always have $k_X\ge 1$, and if $X$ is a contractible topological manifold then $k_X<\infty$ by the Oxtoby--Ulam theorem \cite{Oxtoby:1941aa}. A transport constant $k_X=1$ means that the $W_\infty$ Monge transportation problem (see \cite[Chapter 1]{Villani:2009aa}) admits an approximate continuous solution. The following examples of spaces with this property are collected from \cite[Proposition 2.2, Proposition 2.5 and Theorem 2.13]{Jacelon:2021wa} and \cite[Theorem 2.8]{Jacelon:2021vc}.

\begin{theorem}[\cite{Jacelon:2021wa,Jacelon:2021vc}] \label{thm:tex}
$k_X=1$ if $X$ is:
\begin{enumerate}[(i)]
\item an interval, a circular arc not larger than a semicircle, or a full circle in the plane, equipped with the Euclidean metric;
\item a compact convex subset of Euclidean space, again with the Euclidean metric;
\item a Riemannian manifold of dimension $\ge 3$, equipped with the intrinsic metric.
\end{enumerate}
\end{theorem}

In particular, $k_X=1$ if $X=U(n)$, $X=SU(n)$ or $X=Sp(n)$ for some $n$. The $K$-theory of these Lie groups is also finitely generated and torsion free, so the following implies equality of $d_U$ and $W_\infty$ for unital $^*$-monomorphisms from these commutative compact quantum groups into $A$. One might wonder whether this result could be extended to noncommutative deformations.

\begin{theorem}[\cite{Jacelon:2021vc}] \label{thm:dudw}
Suppose that $k_X<\infty$ and that $K^*(X)$ is finitely generated and torsion free. Let $A$ be a real rank zero object in $\mathcal{E}$ with $\partial_e(T(A))$ nonempty, compact and of finite Lebesgue covering dimension. Then,
\[
W_\infty(\varphi,\psi) \le d_U(\varphi,\psi) \le k_X W_\infty(\varphi,\psi)
\]
for every pair of unital $^*$-monomorphisms $\varphi,\psi\colon C(X)\to A$ with $K_*(\varphi)=K_*(\psi)$.
\end{theorem}

To prove this restatement of \cite[Theorem 3.1]{Jacelon:2021vc}, we must make use of the classification of unital $^*$-monomorphisms from $C(X)$ into $A$. The following is one of the main theorems of \cite{Carrion:wz}. It is also contained in \cite{Gong:2021va} under the assumptions that $B$ is simple and $A$ also satisfies the UCT (which covers our use of the theorem in Section~\ref{section:chaos}). Precursor results including, for example, \cite{Matui:2011uq,Lin:2014aa}, were much more limited in scope.

\begin{theorem}[\cite{Carrion:wz}] \label{thm:morclass1}
Let $B$ be a separable, unital, nuclear $\mathrm{C}^*$-algebra satisfying the UCT and let $A$ be simple, separable, unital, nuclear and $\mathcal{Z}$-stable. Then, for every unital embedding $\varphi\colon B\to A$, compact subset $L\subseteq B$ and $\varepsilon>0$, there is $\delta>0$ such that, if $\varphi'\colon B\to A$ is another unital $^*$-monomorphism that agrees with $\varphi$ on $K$-theory and $\delta$-agrees on traces, then there is a unitary $u\in U(A)$ such that $\sup_{f\in L}\|\varphi(f)-u\varphi'(f)u^*\|<\varepsilon$.
\end{theorem}

We will give a more precise treatment of this somewhat vague statement in the next section. In general, agreement on $K$-theory has to take total $K$-theory and Hausdorffised algebraic $K_1$ into account, but if $A$ has real rank zero and $K_*(B)$ is finitely generated and torsion free, then it simply means agreement in $\mathrm{Hom}(K_*(B),K_*(A))$. If $X$ is a compact, connected metric space, then $W_\infty$ provides a metrisation of the $w^*$-topology on the subspace of faithful measures inside $\mathcal{M}_1^+(X)\cong T(C(X))$ (see \cite[Proposition 2.2]{Jacelon:2021vc}), so when $B=C(X)$, $\delta$-agreement on traces can be interpreted as $W_\infty(\varphi,\varphi')<\delta$. For general $B$, one can replace $W_\infty$ by $W_p$ for $p\in[1,\infty)$ (see \cite[\S2.1]{Jacelon:2021vc}).

\begin{proof}[Proof of Theorem~\ref{thm:dudw}]
We will sketch the argument in the monotracial case $T(A)=\{\tau\}$. The general case is proved similarly after diagonalising with respect to a suitable partition of unity on $\partial_e(T(A))$.

Perturbing if necessary, we may assume that $\mu=\mu_{\varphi^*\tau}$ and $\nu=\mu_{\psi^*\tau}$ are diffuse (faithfulness being automatic since $A$ is simple). Let $h\colon X \to X$ be a homeomorphism with
\begin{itemize}
\item $K_*(h)=\mathrm{id}_{K_*(C(X))}$ (which is the only reason that we want $h$ to be homotopic to $\mathrm{id}_X$),
\item $d(h,\mathrm{id}_X) < k_X W_\infty(\mu,\nu) + \frac{\varepsilon}{2} = k_X W_\infty(\varphi,\psi) + \frac{\varepsilon}{2}$, and
\item $W_\infty(\mu,h_*\nu)<\delta$, where $\delta\in\left(0,\frac{\varepsilon}{2k_X}\right)$ is provided by Theorem~\ref{thm:morclass1} for $B=C(X)$, $L=\mathrm{Lip}^1(X,d)\cap B_{\mathrm{diam}(X)}(C(X))$ and $\frac{\varepsilon}{2k_X}$.
\end{itemize}
Define $\varphi'=\psi\circ h^*\colon C(X) \to A$. Then, $\varphi'$ is a unital $^*$-monomorphism that agrees with $\varphi$ on $K$-theory and with $W_\infty(\varphi,\varphi')<\delta$, so $d_U(\varphi,\varphi')<\frac{\varepsilon}{2k_X}$. Further, since $\varphi'$ and $\psi$ have commuting images, it can be shown that $W_\infty(\varphi',\psi)\le d_U(\varphi',\psi)$ (see \cite[Corollary 3.6]{Jacelon:2021wa}). Combining all of this together gives
\begin{align*}
d_U(\varphi,\psi) &\le d_U(\varphi',\psi) + \frac{\varepsilon}{2}\\
&=\inf_{u\in U(A)}\sup_{f\in\mathrm{Lip}^1(X,d)} \|\varphi'(f)-u\psi(f)u^*\| + \frac{\varepsilon}{2}\\
&\le \sup_{f\in\mathrm{Lip}^1(X,d)}\|\varphi'(f)-\psi(f)\| + \frac{\varepsilon}{2}\\
&= \sup_{f\in\mathrm{Lip}^1(X,d)}\|\psi(f\circ h) - \psi(f)\| + \frac{\varepsilon}{2}\\
&= \sup_{f\in\mathrm{Lip}^1(X,d)}\|f\circ h-f\| + \frac{\varepsilon}{2}\\
&\le k_X W_\infty(\varphi,\psi) + \varepsilon\\
&\le k_X W_\infty(\varphi',\psi) + \frac{3\varepsilon}{2}\\
&\le k_X d_U(\varphi',\psi) + \frac{3\varepsilon}{2}\\
&\le k_X d_U(\varphi,\psi) + 2\varepsilon.
\end{align*}
Since $\varepsilon>0$ is arbitrary, we conclude that
\[
\frac{1}{k_X}d_U(\varphi,\psi) \le W_\infty(\varphi,\psi) \le d_U(\varphi,\psi). \qedhere
\]
\end{proof}

\section{Application B: chaotic tracial dynamics} \label{section:chaos}

In this section, we address the following.

\begin{question} \label{q1}
What is the generic tracial behaviour of automorphisms of stably finite objects in $\mathcal{E}$?
\end{question}

We will successively reduce this initial query until we arrive at one that we can answer using classification and known results from topological dynamics. Throughout, we will assume that $T(A)\ne\emptyset$ is a Bauer simplex (which we recall means that $\partial_e(T(A))$ is compact). Equipped with the topology of pointwise convergence, the set $\mathrm{Aut}(A)$ of $^*$-isomorphisms $A\to A$ is a Polish space. Every $\alpha\in\mathrm{Aut}(A)$ induces an affine homeomorphism $T(\alpha)\in\mathrm{Homeo}_{\mathrm{Aff}}(T(A),T(A))$, namely, $T(\alpha)\colon \tau\mapsto \tau\circ\alpha$. Our question is: what is the typical behaviour of $T(\alpha)$? More precisely:

\begin{question} \label{q2}
What are properties of $T(\alpha)$ that hold residually, that is, for at least a dense $G_\delta$ subset of $\mathrm{Aut}(A)$?
\end{question}

Recall from Section~\ref{section:traces} that $T(A)$ is a Choquet simplex: there is a one-to-one correspondence $T(A) \to \mathcal{M}_1^+(\partial_e(T(A)))$ that sends $\tau$ to $\mu_\tau$, its unique representing measure, which means that $\tau(a)=\int_{\partial_e(T(A))}\widehat{a}\,d\mu_\tau$ for every self-adjoint $a\in A$ (where $\widehat{a}(\sigma)=\sigma(a)$). We denote the inverse by $\mu\mapsto\tau_\mu$. This allows us to identify $\mathrm{Homeo}_{\mathrm{Aff}}(T(A),T(A))$ with $\mathrm{Homeo}(\partial_e(T(A)),\partial_e(T(A)))$ (as topological spaces with the topology of pointwise convergence), since every homeomorphism $h \colon \partial_e(T(A)) \to \partial_e(T(A))$ extends uniquely to an affine homeomorphism $T(A)\to T(A)$ via the pushforward
\[
\sigma\leftrightarrow\mu_\sigma\mapsto h_*\mu_\sigma\leftrightarrow\tau_{h_*\mu_\sigma}.
\]

\begin{question} \label{q3}
What are typical properties of $h=T(\alpha)|_{\partial_e(T(A))}$?
\end{question}

Since for a given $\alpha\in\mathrm{Aut}(A)$, $T(\alpha)$ gives an affine action of the  amenable group $\mathbb{Z}$ on the compact convex set $T(A)$, there is a fixed point, that is, a trace $\tau\in T(A)$ such that $\tau\circ\alpha=\tau$. (Equivalently, there is $\mu=\mu_\tau\in\mathrm{Homeo}(\partial_e(T(A)),\partial_e(T(A)))$ such that $h_*\mu=\mu$.) In other words, writing $\mathrm{Aut}(A,\tau)$ for those automorphisms of $A$ fixing a given $\tau$ (corresponding to $\mathcal{H}(\partial_e(T(A)),\mu_\tau)$, where $\mathcal{H}(X,\mu)$ denotes the $\mu$-preserving homeomorphisms $X\to X$), we have $\mathrm{Aut}(A)=\bigcup_{\tau\in T(A)}\mathrm{Aut}(A,\tau)$.

\begin{question} \label{q4}
What are typical properties in $\mathrm{Aut}(A,\tau)$?
\end{question}

Much can be said about typical properties in $\mathcal{H}(X,\mu)$, for certain $X$ and $\mu$. To make use of this knowledge, we are finally left to ask the following.

\begin{question} \label{q5}
When can we conclude that $T^{-1}(V)$ is a dense $G_\delta$ subset of $\mathrm{Aut}(A,\tau)$ whenever $V$ is a dense $G_\delta$ subset of $\mathcal{H}(\partial_e(T(A)),\mu_\tau)$?
\end{question}

We supply the following answer to Question~\ref{q5}.

\begin{theorem}[\cite{Jacelon:2022wr}] \label{thm:dense}
Let $A\in\mathcal{E}$ with $\partial_e(T(A))\ne\emptyset$ compact, such that the pairing $\rho_A\colon K_0(A) \to \mathrm{Aff}(T(A))$ is trivial (that is, $\widehat{p}\colon T(A)\to\mathbb{R}$ is constant for every projection $p$ over $A$). Then, for every $\tau\in T(A)$ and dense $G_\delta$ subset $V$ of $\mathcal{H}(\partial_e(T(A)),\mu_\tau)$, $T^{-1}(V)$ is a dense $G_\delta$ subset of $\mathrm{Aut}(A,\tau)$.
\end{theorem}

Examples of objects $A$ that satisfy the hypothesis of trivial tracial pairing are limits of subhomogeneous  algebras with connected spectra (assuming simplicity and no dimension growth), in particular, limits of dimension drop algebras $Z_{p_n,q_n}$, interval algebras $C([0,1],M_{k_n})$ and circle algebras $C(S^1,M_{k_n})$. On the other hand, AF algebras (or more generally, stably finite real rank zero objects in $\mathcal{E}$) do not have this property unless they are monotracial, because as mentioned in Section~\ref{section:traces}, $\rho_A(K_0(A))$ is dense in $\mathrm{Aff}(T(A))$ for such a $\mathrm{C}^*$-algebra $A$.

\subsection{The proof of Theorem~\ref{thm:dense}: lifting via classification}

Fix $\alpha\in\mathrm{Aut}(A,\tau)$, a finite set $F\subseteq A$ and $\varepsilon>0$. Let $\mathrm{Inv}$ be an invariant based on $K$-theory and traces that classifies morphisms $A\to A$ (which was alluded to in Theorem~\ref{thm:morclass1}), meaning:
\begin{itemize}
\item (existence) every morphism $\mathrm{Inv}(A)\to\mathrm{Inv}(A)$ can be lifted to a morphism $A\to A$;
\item (uniqueness) for every morphism $\varphi\colon A\to A$, there is $\delta>0$ such that, if $\psi\colon A\to A$ is a morphism for which $\mathrm{Inv}(\psi)$ agrees with $\mathrm{Inv}(\varphi)$ on $K$-theory and $\delta$-agrees on traces (that is, for which $T(\psi)$ and $T(\varphi)$ are uniformly close, within $\delta$, with respect to some fixed $w^*$-metrisation of $T(A)$), then there is a unitary $u\in U(A)$ with $\|\varphi(a)-u\psi(a)u^*\| < \varepsilon$ for every $a\in F$.
\end{itemize}
Given such an invariant, the strategy is to:
\begin{enumerate}[1.]
\item perturb $\mathrm{Inv}(\alpha)$ by keeping its $K$-theory part the same but replacing its tracial part $T(\alpha)$ by a nearby $h\in\mathcal{V}$;
\item use existence to lift this perturbed $\mathrm{Inv}$-morphism to a morphism $\alpha_h\colon A\to A$;
\item use uniqueness to deduce that a unitary conjugate $\beta$ of $\alpha_h$ is $\varepsilon$-close to $\alpha$ on $F$.
\end{enumerate}
For this to work, we must know that the perturbation in the first step still gives a valid $\mathrm{Inv}$-morphism. So, what is $\mathrm{Inv}$?

Recall that agreement on $\mathrm{Ell}$ is a necessary condition for approximate unitary equivalence on $^*$-homomorphisms. But it is not sufficient, as the following example shows. Let $\tau$ be the unique trace on the Jiang--Su algebra $\mathcal{Z}$, $\lambda$ the trace on $C(S^1)$ corresponding to Lebesgue measure on $S^1$, $\alpha\colon C(S^1)\to\mathcal{Z}$ a unital $^*$-homomorphism with $\alpha^*\tau=\lambda$ (which exists by \cite[Theorem 2.1]{Rordam:2004kq}), $\rho_\theta\colon C(S^1) \to C(S^1)$ the automorphism induced by rotation by $\theta\in(0,2\pi)$ and $\beta=\alpha\circ\rho_\theta$. Then, $\mathrm{Ell}(\alpha)=\mathrm{Ell}(\beta)$ but $\alpha$ and $\beta$ are not approximately unitarily equivalent. What distinguishes them is the de la Harpe--Skandalis determinant \cite{Harpe:1984aa}, or equivalently, \emph{Hausdorffised unitary algebraic $K_1$}
\[
\overline{K_1}^{alg}(A):=U_\infty(A)/\overline{DU_\infty}(A),
\]
where $\overline{DU_\infty}(A)$ denotes the closure (in the inductive limit topology on $U_\infty(A) = \bigcup_{n\in\mathbb{N}}U_n(A)$) of the derived subgroup of $U_\infty(A)$ (that is, the subgroup generated by commutators). $\overline{K_1}^{alg}(A)$ is related to $\mathrm{Ell}$ via the Thomsen exact sequence
\[
\begin{tikzcd}
\text{\cite{Thomsen:1995qf}:} \quad 0 \arrow[r] & \mathrm{Aff}(T(A))/\overline{\rho_A(K_0(A))} \arrow[r,"\lambda_A"] & \overline{K_1}^{alg}(A) \arrow[r,"\pi_A"] & K_1(A) \arrow[r] & 0,
\end{tikzcd}
\]
where $\lambda_A$ is the inverse of the determinant and $\pi_A([u]_{alg})=[u]_1$.

There is one remaining component of $\mathrm{Inv}$. If there is torsion in $K_*(A)$, then one must also include \emph{total $K$-theory} $\underline{K}(A) = \bigoplus_{n=0}^\infty K_*(A;\mathbb{Z}/n)$ with associated `Bockstein maps'. But that is enough: $\mathrm{Inv}$ is the refinement of $\mathrm{Ell}$ that includes $\overline{K_1}^{alg}$ and $\underline{K}$, together with compatibility maps that must be respected by morphisms, namely

\begin{equation} \label{eqn:compatible1}
\begin{tikzcd}
K_0(A) \arrow[r,"\rho_A"] \arrow[d] & \mathrm{Aff}(T(A)) \arrow[r,"\lambda_A"] \arrow[d] & \overline{K_1}^{alg}(A) \arrow[r,"\pi_A"] \arrow[d] & K_1(A) \arrow[d]\\
K_0(A) \arrow[r,"\rho_A"] & \mathrm{Aff}(T(A)) \arrow[r,"\lambda_A"] & \overline{K_1}^{alg}(A) \arrow[r,"\pi_A"] & K_1(A)
\end{tikzcd}
\end{equation}
and another family of commutative diagrams between $\underline{K}$ and $\overline{K_1}^{alg}$ that is identified in \cite{Carrion:wz}. To finish the proof of Theorem~\ref{thm:dense}, we must show that the perturbation $\alpha_h$ meets the compatibility requirements. The first, \eqref{eqn:compatible1}, is automatic under the assumption of trivial tracial pairing. The second can be arranged by an adjustment of $\overline{K_1}^{alg}(\alpha_h)$ that does not affect $T(\alpha_h)$ (and in fact is automatic if $K_1(A)$ is torsion free). This completes the proof.

\subsection{Dynamics on manifolds}

For the rest of this section, $X$ is a compact topological manifold, $\mu\in\mathcal{M}_1^+(X)$ is an \emph{OU measure}, that is, $\mu$ is faithful, diffuse and zero on the boundary $\partial X$ (if there is one), and $h\in\mathcal{H}(X,\mu)$. The name `OU measure' is in reference to the Oxtoby--Ulam theorem, which identifies OU measures on $[0,1]^n$ as precisely the measures that are homeomorphic images of Lebesgue measure. In the boundaryless case, OU measures form a dense $G_\delta$ subset of $\mathcal{M}_1^+(X)$, so the restriction of our attention to them is not that unreasonable.

\begin{definition}
The system $(X,h)$ is \emph{chaotic} if $h$ is topologically transitive (that is, there is a dense orbit) and the set of periodic points of $h$ is dense in $X$.
\end{definition}

By \cite{Banks:1992ux}, these conditions imply \emph{sensitive dependence on initial conditions}: there exists $\delta>0$ such that, for every $x\in X$ and every $\varepsilon>0$, there exist $y\in X$ and $n\in\mathbb{N}$ such that $d(x,y)<\varepsilon$ and $d(h^nx,h^ny)\ge\delta$.

\begin{definition}
The $\mu$-preserving map $h\colon X\to X$ is:
\begin{enumerate}[1.]
\item \emph{ergodic} if
\[
\lim_{n\to\infty}\frac{1}{n}\sum_{k=0}^{n-1}\frac{\mu(h^{-k}(U)\cap V)}{\mu(U)} = \mu(V);
\]
\item \emph{weakly mixing} if
\[
\lim_{n\to\infty}\frac{1}{n}\sum_{k=0}^{n-1}\left|\frac{\mu(h^{-k}(U)\cap V)}{\mu(U)} - \mu(V)\right| = 0;
\]
\item \emph{strongly mixing} if
\[
\lim_{n\to\infty}\frac{\mu(h^{-n}(U)\cap V)}{\mu(U)} = \mu(V),
\]
the equations holding for all (non-null) measurable sets $U,V$.
\end{enumerate}
\end{definition}

As described in \cite{Halmos:1960vh}, if $h$ represents the stirring of a martini (containing, say, ninety per cent gin and ten per cent vermouth, the latter initially occupying a region $V$), then the left-hand side of these equations measures the proportion of vermouth in a given region $U$ after many stirs. Ergodicity means that, on average, the proportion of vermouth in $U$ is the desired ten percent. Strong mixing is the stronger property of eventually having a perfectly mixed drink (no average necessary). Weak mixing is the intermediate property of eventually having the right amount of vermouth in $U$, except for a few rare instances (specifically in a set of integers of density zero) of the drink being either too strong or too sweet.

Irrational rotation of the circle (with Lebesgue measure) is an example of an ergodic map that is not weakly mixing. Since it is also minimal (so in particular, there are no periodic points), it is not chaotic either. A strongly mixing and chaotic example is Arnold's cat map (a volume-preserving diffeomorphism of the torus), or more generally, a mixing Anosov system. For more examples, see \cite{Jacelon:2022wr}.

\begin{theorem} \label{thm:dynamics}
Let $X$ be a compact, connected topological manifold and $\mu$ an OU measure on $X$. Then, the typical element of $\mathcal{H}(X,\mu)$ is:
\begin{enumerate}[(i)]
\item \cite{Katok:1970wx} weakly mixing, and
\item \cite{Aarts:1999wc,Daalderop:2000wm} chaotic if $\dim X\ge 2$.
\end{enumerate}
\end{theorem}

Combining Theorem~\ref{thm:dense} and Theorem~\ref{thm:dynamics} provides us with a response to Question~\ref{q4}, hence to Question~\ref{q1}. Remember that any metrisable Choquet simplex can be realised as the trace space of objects in $\mathcal{E}$ (in particular, limits of dimension drop algebras), so there are many examples covered by the following.

\begin{corollary}[\cite{Jacelon:2022wr}]
Let $A\in\mathcal{E}$ such that $\partial_e(T(A))\ne\emptyset$ is a compact, connected topological manifold and such that the pairing $\rho_A\colon K_0(A) \to \mathrm{Aff}(T(A))$ is trivial, and let $\tau\in T(A)$ be represented by an OU measure $\mu$. Then, the typical $\tau$-preserving automorphism $\alpha$ of $A$ induces $h\in\mathcal{H}(\partial_e(T(A)),\mu)$ that is weakly mixing, and is also chaotic if $\dim X\ge 2$. \qed
\end{corollary}

\section{Application C: random $\mathrm{C}^*$-algebras} \label{section:random}

Imagine that one constructs a $\mathrm{C}^*$-algebra in one's favourite way, perhaps as an inductive limit or graph algebra, but instead of making specific choices to obtain a certain kind of structure, one makes \emph{random} choices. Then, instead of asking whether the algebra has that structure, one asks \emph{how likely} the structure is. Through classification, we can give this thought experiment mathematical meaning. The study of random graph algebras is initiated in \cite[\S6]{Jacelon:2023aa} and continued in \cite{Jacelon:2023ab}. Here, we consider inductive limits, specifically limits of prime dimension drop algebras so that the invariant is purely tracial. We use the Lazar--Lindenstrauss theorem mentioned in Section~\ref{section:traces} to build random simplices and therefore (the isomorphism classes of) random limits $\varinjlim Z_{p_n,q_n}$.

We introduce the following terminology in order to provide some extrinsic justification for focusing on these particular limits.

\begin{definition}
Say that a separable $\mathrm{C}^*$-algebra $A$ is \emph{$K$-contractible} if it is $KK$-equivalent to $\mathbb{C}$, which in the presence of the UCT is equivalent to $K_*(A) \cong K_*(\mathbb{C})$, that is, $K_0(A) \cong \mathbb{Z}$ and $K_1(A)=0$. Say that $A\in\mathcal{E}$ is \emph{strongly $K$-contractible} if
\[
(K_0(A),K_0(A)_+,[1_A],K_1(A)) \cong (\mathbb{Z},\mathbb{N}_0,1,0).
\]
\end{definition}

In this language, the Jiang--Su algebra $\mathcal{Z}$ is the unique monotracial strongly $K$-contractible object in $\mathcal{E}$.

\begin{question} \label{q7}
What is $\mathbb{P}\left(A\cong \mathcal{Z} \mid A\in\mathcal{E} \text{ is strongly $K$-contractible}\right)$?
\end{question}

As alluded to earlier, we will make this precise by introducing probability distributions on the collection of metrisable Choquet simplices (the trace space being the only component of the Elliott invariant not determined by the notion of strong $K$-contractibility). First, we reiterate that it suffices to consider limits of prime dimension drop algebras.

\begin{theorem}
If $A\in\mathcal{E}$ is $K$-contractible, then there exists $k\in\mathbb{N}$ such that either
\begin{enumerate}[(i)]
\item \label{prob1} $A \cong M_k(\mathcal{O}_\infty)$, or
\item \label{prob2} $A \cong \varinjlim M_k(Z_{p_n,q_n})$ for some coprime natural numbers $p_n,q_n$.
\end{enumerate}
If $A\in\mathcal{E}$ is strongly $K$-contractible, then \eqref{prob2} holds with $k=1$.
\end{theorem}

\begin{proof}
This follows from classification, since the model algebras stated in the theorem exhaust the possibilities for the invariant. The number $k$ represents the class of the unit in $K_0(A) \cong \mathbb{Z}$. If $A$ is purely infinite, then $\mathrm{Ell}$ reduces to
\[
(K_0(A),[1_A],K_1(A)) \cong (\mathbb{Z},k,0) \cong (K_0(M_k(\mathcal{O}_\infty)),[1_{M_k(\mathcal{O}_\infty)}],K_1(M_k(\mathcal{O}_\infty))).
\]
If $A$ is stably finite, then $(K_0(A),K_0(A)_+,[1_A])$ is a weakly unperforated ordered group, and the order structure is thus fixed by $K$-contractibility. To see this, note that $\tau(1)=1>0$ for every $\tau\in T(A)$, so by weak unperforation (since every state on $K$-theory is induced by a trace), $k\cdot1=[1_A]\in K_0(A)_+$ and therefore $1\in K_0(A)_+$. It follows that $K_0(A)_+\cong\mathbb{N}_0$. So by \cite[Theorem 4.5]{Jiang:1999hb}, there exists a simple inductive limit $B$ of prime dimension drop algebras such that $\mathrm{Ell}(M_k(B))\cong\mathrm{Ell}(A)$.
\end{proof}

\subsection{Representing matrices} \label{subsection:repmat}

Let $\Delta$ be a metrisable Choquet simplex. By the Lazar--Lindenstrauss theorem \cite[Theorem 5.2 and its Corollary]{Lazar:1971kx}, we can write $\Delta = \varprojlim(\Delta_n,\psi_n)$, where $\Delta_n=\mathrm{conv}\{e_{0,n},\dots,e_{n,n}\}$ sits atop its base $\Delta_{n-1}$, and $\psi\colon \Delta_n \to \Delta_{n-1}$ is the affine `collapsing' map that fixes the base and maps $e_{n,n}$ to some point with barycentric coordinates $(a_{1,n},\dots,a_{n,n})$, that is, $\psi_n(e_{n,n})=\sum_{i=1}^na_{i,n}e_{i-1,n-1}$. The upper triangular matrix $(a_{ij})_{i,j\in\mathbb{N}}$ is called a \emph{representing matrix} for $\Delta$. We identify the collection of all representing matrices with $\prod_{n=0}^\infty\Delta_n$ as sets and indeed as measure spaces once we fix a Borel product measure $\mu=\bigotimes_{n=0}^\infty\mu_n$ on $\prod_{n=0}^\infty\Delta_n$. Note that the map that assigns to a representing matrix the Choquet simplex built from the associated sequence of finite-dimensional simplices is not one to one. Nevertheless, pushing forward does give us a probability measure on the set of (isomorphism classes of) simplices.

We will consider three candidates for $\mu_n$:
\begin{enumerate}[(I)]
\item $\mu_n$ the point mass at the barycentre $\frac{1}{n+1}\left(e_{0,n}+\dots+e_{n,n}\right)$ of $\Delta_n$;
\item  $\mu_n$ the uniform measure on the $n+1$ vertices of $\Delta_n$;
\item $\mu_{\frac{n(n+1)}{2}+i}$ normalised Lebesgue measure on $\mathrm{conv}\{e_{0,n},\dots,e_{n-i+1,n}\}\subseteq\Delta_n\subseteq\Delta_{\frac{n(n+1)}{2}+i}$ for $1\le i\le n+1$ (that is, for every $\Delta_n$ we take successive uniform samples of the faces $\Delta_n\supseteq\Delta_{n-1}\supseteq\dots\supseteq\Delta_0$).
\end{enumerate}

Actually, we allow for \emph{two} source of randomness:
\begin{itemize}
\item a random choice of collapse points (described above);
\item a random sequence of simplex dimensions (described below).
\end{itemize}

\subsection{Random walks}

Let $(Y_n)_{n=0}^\infty$ be a simple random walk on $\mathbb{N}_0$ with initial distribution $\pi=(\pi_i)_{i=0}^\infty$ and transition matrix $\Pi$, that is, for any $i,j,n\ge0$, we have $\mathbb{P}(Y_0=i)=\pi_i$ and
\[
 \mathbb{P}(Y_{n+1}=j \mid Y_n = i) = \Pi_{ij} =
 \begin{cases}
 p & \text{ if }\: j=i+1\\
 q & \text{ if }\: j=i-1\\
 0 & \text{ if }\: |i-j|>1.
 \end{cases}
\]
At $0$, there is a reflecting barrier: $\mathbb{P}(Y_{n+1}=1 \mid Y_n = 0)=1$. The walk is represented by the diagram
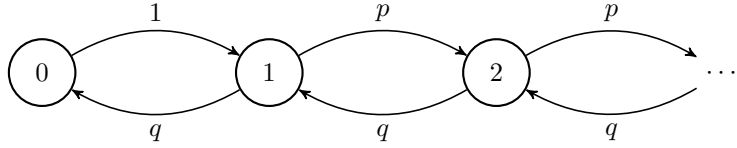
\[
	\begin{tikzpicture}[->, >=stealth', auto, semithick, node distance=3cm]
	
	\node[state][draw=black,thick,text=black,scale=1]    (B)[]   {$0$};
	\node[state][draw=black,thick,text=black,scale=1]    (C)[right of=B]   {$1$};
	\node[state][draw=black,thick,text=black,scale=1]    (D)[right of=C]   {$2$};
	\node   (E)[right of=D]   {$\cdots$};
	\path
	(B) edge[bend left,above]		node{$1$}	(C)
	(C) edge[bend left,below]	node{$q$}	(B)
	edge[bend left,above]		node{$p$}	(D)
	(D) edge[bend left,below]	node{$q$}	(C)
	edge[bend left,above]		node{$p$}       (E)
	(E) edge[bend left,below]	node{$q$}	(D);
	\end{tikzpicture}
\]
This is a special case of a (discrete-time) \emph{Markov chain}, meaning that only the current state of the process affects the next transition, or in other words that the process does not retain any memory of the past. It is also \emph{irreducible}, which means that, for any states $i$ and $j$, there exists $n\ge0$ such that $\mathbb{P}(Y_n=j \mid Y_0=i) >0$. A state $i$ is called
\begin{itemize}
\item \emph{recurrent} if with probability $1$ it is visited infinitely often, and
\item \emph{transient} if with probability $1$ it is visited only finitely many times.
\end{itemize}
Every state is either recurrent or transient (see \cite[Theorem 1.5.3]{Norris:1998wc}) and for irreducible chains, either every state is recurrent or every state is transient (see \cite[Theorem 1.5.4]{Norris:1998wc}). It can be shown that our random walk $(Y_n)_{n=0}^\infty$ on $\mathbb{N}_0$ is recurrent if $p\le q$ and transient if $p>q$. See the discussion and references in \cite[\S2]{Jacelon:2023aa}.

\begin{remark}
\begin{enumerate}[1.]
\item With the present setup, the walk is always infinite. We can allow for the possibility of finite walks by introducing an `absorbing state' at $0$.
\item In higher dimensions, a simple symmetric walk on the grid $\mathbb{Z}^d$ is recurrent for $d=2$ but transient for $d\ge3$ (see, for example, \cite[\S1.6]{Norris:1998wc}).
\end{enumerate}
\end{remark}

\subsection{Random Choquet simplices and random $\mathrm{C}^*$-algebras}

We construct a random projective limit of finite-dimensional simplices
\[
\Delta(\Pi,\pi,\mu) = \varprojlim(\Delta_{Y_n},\psi_n\colon\Delta_{Y_n}\to\Delta_{Y_{n-1}}),
\]
where $(Y_n)$ determines the dimension of the simplex and $\mu=\bigotimes_{n=0}^\infty\mu_n$ determines the connecting map $\psi_n$, namely: $\psi_n$ is the base inclusion if $Y_{n}=Y_{n-1}-1$ and is a random collapsing map chosen  according to $\mu$ if $Y_{n}=Y_{n-1}+1$. Then, we associate to $\Delta(\Pi,\pi,\mu)$ the (isomorphism class of the) simple inductive limit $Z(\Pi,\pi,\mu)$ of prime dimension drop algebras that has this simplex as its trace space. This is our random strongly $K$-contractible object in $\mathcal{E}$.

\begin{theorem}[\cite{Jacelon:2023aa}]
\begin{enumerate}[1.]
\item If $p\le q$, then $\mathbb{P}\left(Z(\Pi,\pi,\mu)\cong\mathcal{Z}\right)=1$.
\item If $p>q$, then $\mathbb{P}\left(\dim T(Z(\Pi,\pi,\mu))=\infty\right)=1$. Specifically, $T(Z(\Pi,\pi,\mu))$ is almost surely affinely homeomorphic to:
\begin{enumerate}[(I)]
\item the Bauer simplex with boundary homeomorphic to $\{\frac{1}{n} \mid n\in\mathbb{N}\} \cup \{0\}$, or
\item the Bauer simplex with boundary homeomorphic to the Cantor set, or
\item the Poulsen simplex
\end{enumerate}
depending on whether $\mu$ is of the form (I), (II) or (III) outlined in Section~\ref{subsection:repmat}.
\end{enumerate}
\end{theorem}

\begin{proof}
In the first case, the $0$-simplex is almost surely visited infinitely often, so the associated inductive sequence is trace collapsing and therefore yields $\mathcal{Z}$. In the second case, the walk is transient and so, almost surely, $\dim \Delta_{Y_n}$ wanders off to $\infty$. The statements about the kind of simplex we get (depending on the choice of measure $\mu$) follow from the theory of representing matrices (see \cite[\S3]{Jacelon:2023aa}). 
\end{proof}

\begin{remark}
With an absorbing state at $0$, we can allow for $T(Z(\Pi,\pi,\mu))$ to be a nontrivial finite-dimensional simplex with positive probability. In this case, we would take the dimension of the simplex to be $\sup\{Y_n\mid n\ge0\}$ and could compute the probability of this dimension being at most some given $k$. See \cite{Jacelon:2023aa}.
\end{remark}

\subsection*{Acknowledgements} The author was supported by the Czech Science Foundation (GA\v{C}R) project 22-07833K and the Institute of Mathematics of the Czech Academy of Sciences (RVO: 67985840), and partially supported by the Simons Foundation Award No 663281 granted to the Institute of Mathematics of the Polish Academy of Sciences for the years 2021--2023.


\begin{thebibliography}{10}

\bibitem{Aarts:1999wc}
J.~M. Aarts and F.~G.~M. Daalderop,
  {\em Chaotic homeomorphisms on manifolds},
  Topology Appl. 96, 1 (1999), 93--96.

\bibitem{Alfsen:1971hl}
E.~M. Alfsen,
  {\em Compact convex sets and boundary integrals},
  Springer-Verlag, New York, 1971,
  Ergebnisse der Mathematik und ihrer Grenzgebiete, Band 57.

\bibitem{Antoine:2011zl}
R.~Antoine, F.~Perera, and L.~Santiago,
  {\em Pullbacks, {${\rm C}(X)$}-algebras, and their {C}untz semigroup},
  J. Funct. Anal. 260, 10 (2011), 2844--2880.

\bibitem{Ara:2009cs}
P.~Ara, F.~Perera, and A.~S. Toms,
  {\em {$K$}-theory for operator algebras. {C}lassification of
  {$\mathrm{C}^*$}-algebras},
  in: Aspects of operator algebras and applications, vol.~534 of
  Contemp. Math., Amer. Math. Soc., Providence, RI, 2011, pp.~1--71.

\bibitem{Banks:1992ux}
J.~Banks, J.~Brooks, G.~Cairns, G.~Davis, and P.~Stacey,
  {\em On {D}evaney's definition of chaos},
  Amer. Math. Monthly 99, 4 (1992), 332--334.

\bibitem{Bates:2000fk}
T.~Bates, D.~Pask, I.~Raeburn, and W.~Szyma{{\'n}}ski,
  {\em The {$\mathrm{C}^*$}-algebras of row-finite graphs},
  New York J. Math. 6\/ (2000), 307--324 (electronic).

\bibitem{Bhatia:1984aa}
R.~Bhatia and C.~Davis,
  {\em A bound for the spectral variation of a unitary operator},
  Linear and Multilinear Algebra 15, 1 (1984), 71--76.

\bibitem{Blackadar:1998qf}
B.~Blackadar,
  {\em {$K$}-theory for operator algebras}, second~ed., vol.~5 of 
  Mathematical Sciences Research Institute Publications,
  Cambridge University Press, Cambridge, 1998.

\bibitem{Blackadar:1980zr}
B.~Blackadar,
  {\em Traces on simple {AF} {$\mathrm{C}^*$}-algebras},
  J. Funct. Anal. 38, 2 (1980), 156--168.

\bibitem{Bratteli:1972fj}
O.~Bratteli,
  {\em Inductive limits of finite dimensional {$\mathrm{C}^*$}-algebras},
  Trans. Amer. Math. Soc. 171\/ (1972), 195--234.

\bibitem{Brown:2008mz}
N.~P. Brown, F.~Perera, and A.~S. Toms,
  {\em The {C}untz semigroup, the {E}lliott conjecture, and dimension
  functions on {$\mathrm{C}^*$}-algebras},
  J. Reine Angew. Math. 621\/ (2008), 191--211.

\bibitem{Brown:2007rz}
N.~P. Brown and A.~S. Toms,
  {\em Three applications of the {C}untz semigroup},
  Int. Math. Res. Not. IMRN, 19 (2007), Art. ID rnm068, 14.

\bibitem{Carrion:wz}
J.~R. Carri\'{o}n, J.~Gabe, C.~Schafhauser, A.~Tikuisis, and S.~White,
  {\em Classification of $^*$-homomorphisms {{\rm I}}: {S}imple nuclear
  {$\mathrm{C}^*$}-algebras},
  arXiv:2307.06480 [math.OA], 2023.

\bibitem{Castillejos:2020wv}
J.~Castillejos and S.~Evington,
  {\em Nuclear dimension of simple stably projectionless {$\mathrm{C}^*$}-algebras},
  Anal. PDE 13, 7 (2020), 2205--2240.

\bibitem{Castillejos:2021wm}
J.~Castillejos, S.~Evington, A.~Tikuisis, S.~White, and W.~Winter,
  {\em Nuclear dimension of simple {$\mathrm{C}^*$}-algebras},
  Invent. Math. 224, 1 (2021), 245--290.

\bibitem{Ciuperca:2008rz}
A.~Ciuperca and G.~A. Elliott,
  {\em A remark on invariants for {$\mathrm{C}^*$}-algebras of stable rank one},
  Int. Math. Res. Not. IMRN, 5 (2008), Art. ID rnm 158, 33.

\bibitem{Ciuperca:2011wd}
A.~Ciuperca, G.~A. Elliott, and L.~Santiago,
  {\em On inductive limits of type-{I} {$\mathrm{C}^*$}-algebras with one-dimensional
  spectrum},
  Int. Math. Res. Not. IMRN, 11 (2011), 2577--2615.

\bibitem{Conley:2018vn}
C.~Conley and A.~T\"{o}rnquist,
  {\em A {F}ra\"{i}ss\'{e} approach to the {P}oulsen simplex},
  in: Sets and computations, vol.~33 of Lect. Notes Ser.
  Inst. Math. Sci. Natl. Univ. Singap., World Sci. Publ., Hackensack, NJ, 2018,
  pp.~11--24.

\bibitem{Cuntz:1978rz}
J.~Cuntz,
  {\em Dimension functions on simple {$\mathrm{C}^*$}-algebras},
   Math. Ann. 233, 2 (1978), 145--153.

\bibitem{Daalderop:2000wm}
F.~Daalderop and R.~Fokkink,
  {\em Chaotic homeomorphisms are generic},
  Topology Appl. 102, 3 (2000), 297--302.

\bibitem{Harpe:1984aa}
P.~de~la Harpe and G.~Skandalis,
  {\em D\'{e}terminant associ\'{e} {\`a} une trace sur une alg\'{e}bre de
  {B}anach},
  Ann. Inst. Fourier (Grenoble) 34, 1 (1984), 241--260.

\bibitem{Dykema:1997aa}
K.~Dykema, U.~Haagerup, and M.~R{\o}rdam,
  {\em The stable rank of some free product {$\mathrm{C}^*$}-algebras},
  Duke Math. J. 90, 1 (1997), 95--121.

\bibitem{Eilers:1998yu}
S.~Eilers, T.~A. Loring, and G.~K. Pedersen,
  {\em Stability of anticommutation relations: an application of
  noncommutative {CW} complexes},
  J. Reine Angew. Math. 499\/ (1998), 101--143.

\bibitem{Eilers:2021aa}
S.~Eilers, G.~Restorff, E.~Ruiz, and A.~P.~W. S{\o}rensen,
  {\em The complete classification of unital graph {$\mathrm{C}^*$}-algebras:
  geometric and strong},
  Duke Math. J. 170, 11 (2021), 2421--2517.

\bibitem{Elliott:1976kq}
G.~A. Elliott,
  {\em On the classification of inductive limits of sequences of semisimple
  finite-dimensional algebras},
  J. Algebra 38, 1 (1976), 29--44.

\bibitem{Elliott:1993ai}
G.~A. Elliott,
  {\em A classification of certain simple {$\mathrm{C}^*$}-algebras},
  in: Quantum and non-commutative analysis ({K}yoto, 1992),
  vol.~16 of Math. Phys. Stud., Kluwer Acad. Publ., Dordrecht, 1993,
  pp.~373--385.

\bibitem{Elliott:1993kq}
G.~A. Elliott,
  {\em On the classification of {$\mathrm{C}^*$}-algebras of real rank zero},
  J. Reine Angew. Math. 443\/ (1993), 179--219.

\bibitem{Elliott:2016ab}
G.~A. Elliott, G.~Gong, H.~Lin, and Z.~Niu,
  {\em On the classification of simple amenable {$\mathrm{C}^*$}-algebras with finite
  decomposition rank, {\rm{II}}},
  arXiv:1507.03437 [math.OA], 2016.

\bibitem{Elliott:2009kq}
G.~A. Elliott, L.~Robert, and L.~Santiago,
  {\em The cone of lower semicontinuous traces on a {$\mathrm{C}^*$}-algebra},
  Amer. J. Math. 133, 4 (2011), 969--1005.

\bibitem{Farah:2014vt}
I.~Farah, A.~S. Toms, and A.~T\"{o}rnquist,
  {\em Turbulence, orbit equivalence, and the classification of nuclear
  {$\mathrm{C}^*$}-algebras},
  J. Reine Angew. Math. 688\/ (2014), 101--146.

\bibitem{Fu:2022aa}
X.~Fu, K.~Li, and H.~Lin,
  {\em Tracial approximate divisibility and stable rank one},
   J. Lond. Math. Soc. (2) 106, 4 (2022), 3008--3042.

\bibitem{Gardella:2024aa}
E.~Gardella, S.~Geffen, R.~Gesing, Kopsacheilis, and P.~Naryshkin,
  {\em Essential freeness, allostery and {$\mathcal{Z}$}-stability of
  crossed products},
  arXiv:2405.04343 [math.OA], 2024.

\bibitem{Gardella:2023aa}
E.~Gardella, S.~Geffen, J.~Kranz, and P.~Naryshkin,
  {\em Classifiability of crossed products by nonamenable groups},
  J. Reine Angew. Math. 797\/ (2023), 285--312.

\bibitem{Gardella:2022aa}
E.~Gardella and F.~Perera,
  {\em The modern theory of {C}untz semigroups of {$\mathrm{C}^*$}-algebras},
  EMS Surv. Math. Sci. (2024), published online first.

\bibitem{Giol:2010vs}
J.~Giol and D.~Kerr,
  {\em Subshifts and perforation},
   J. Reine Angew. Math. 639\/ (2010), 107--119.

\bibitem{Glimm:1960qh}
J.~G. Glimm,
  {\em On a certain class of operator algebras},
  Trans. Amer. Math. Soc. 95\/ (1960), 318--340.

\bibitem{Gong:2000kq}
G.~Gong, X.~Jiang, and H.~Su,
  {\em Obstructions to {$\mathcal{Z}$}-stability for unital simple
  {$\mathrm{C}^*$}-algebras},
  Canad. Math. Bull. 43, 4 (2000), 418--426.

\bibitem{Gong:2020ud}
G.~Gong, H.~Lin, and Z.~Niu,
  {\em A classification of finite simple amenable {$\mathcal Z$}-stable
  {$\mathrm{C}^*$}-algebras, {I}: {$\mathrm{C}^*$}-algebras with generalized tracial rank
  one},
  C. R. Math. Acad. Sci. Soc. R. Can. 42, 3 (2020), 63--450.

\bibitem{Gong:2020uf}
G.~Gong, H.~Lin, and Z.~Niu,
  {\em A classification of finite simple amenable {$\mathcal Z$}-stable
  {$\mathrm{C}^*$}-algebras, {II}: {$\mathrm{C}^*$}-algebras with rational
  generalized tracial rank one},
  C. R. Math. Acad. Sci. Soc. R. Can. 42, 4 (2020), 451--539.

\bibitem{Gong:2021va}
G.~Gong, H.~Lin, and Z.~Niu,
  {\em Homomorphisms into simple {$\mathcal{Z}$}-stable {$\mathrm{C}^*$}-algebras,
  {II}},
  J. Noncommut. Geom. 17, 3 (2023), 835--898.

\bibitem{Goodearl:1977dq}
K.~R. Goodearl,
  {\em Algebraic representations of {C}hoquet simplexes},
  J. Pure Appl. Algebra 11, 1--3 (1977/78), 111--130.

\bibitem{Goodearl:1992hl}
K.~R. Goodearl,
  {\em Notes on a class of simple {$\mathrm{C}^*$}-algebras with real rank zero},
  Publ. Mat. 36, 2A (1992), 637--654.

\bibitem{Halmos:1960vh}
P.~R. Halmos,
  {\em Lectures on ergodic theory},
  Chelsea Publishing Co., New York, 1960.

\bibitem{Hiai:1989aa}
F.~Hiai and Y.~Nakamura,
  {\em Distance between unitary orbits in von {N}eumann algebras},
  Pacific J. Math. 138, 2 (1989), 259--294.

\bibitem{Higson:1990qd}
N.~Higson,
  {\em A primer on {$KK$}-theory},
  in: Operator theory: operator algebras and applications, Part 1
  (Durham, NH, 1988), vol.~51 of Proc. Sympos. Pure Math., Amer. Math.
  Soc., Providence, RI, 1990, pp.~239--283.

\bibitem{Higson:2000to}
N.~Higson and J.~Roe,
  {\em Analytic {$K$}-homology},
  Oxford Mathematical Monographs, Oxford University Press, Oxford,
  2000.

\bibitem{Hirshberg:2017aa}
I.~Hirshberg and J.~Wu,
  {\em The nuclear dimension of {$\mathrm{C}^*$}-algebras associated to
  homeomorphisms},
  Adv. Math. 304\/ (2017), 56--89.

\bibitem{Holbrook:1992aa}
J.~A. Holbrook,
  {\em Spectral variation of normal matrices},
  Linear Algebra Appl. 174\/ (1992), 131--141.

\bibitem{Jacelon:2022wr}
B.~Jacelon,
  {\em Chaotic tracial dynamics},
  Forum Math. Sigma 11\/ (2023), Paper No. e39, 21.

\bibitem{Jacelon:2021vc}
B.~Jacelon,
  {\em Metrics on trace spaces},
  J. Funct. Anal. 285, 4 (2023), Paper No. 109977, 32.

\bibitem{Jacelon:2023aa}
B.~Jacelon,
  {\em Random amenable {$\mathrm{C}^*$}-algebras},
  Math. Proc. Cambridge Philos. Soc. 175, 2 (2023), 345--366.

\bibitem{Jacelon:2023ab}
B.~Jacelon and I.~Khavkine,
  {\em Operator {$K$}-theoretic analysis of random adjacency matrices},
  New York J. Math. 31 (2025), 749--791.

\bibitem{Jacelon:2014aa}
B.~Jacelon, K.~R. Strung, and A.~S. Toms,
  {\em Unitary orbits of self-adjoint operators in simple
  {$\mathcal{Z}$}-stable {$\mathrm{C}^*$}-algebras},
  J. Funct. Anal. 269, 10 (2015), 3304--3315.

\bibitem{Jacelon:2021wa}
B.~Jacelon, K.~R. Strung, and A.~Vignati,
  {\em Optimal transport and unitary orbits in {$\mathrm{C}^*$}-algebras},
  J. Funct. Anal. 281, 5 (2021), 109068.

\bibitem{Jiang:1999hb}
X.~Jiang and H.~Su,
  {\em On a simple unital projectionless {$\mathrm{C}^*$}-algebra},
  Amer. J. Math. 121, 2 (1999), 359--413.

\bibitem{Katok:1970wx}
A.~B. Katok and A.~M. Stepin,
  {\em Metric properties of homeomorphisms that preserve measure},
  Uspehi Mat. Nauk 25, 2 (152) (1970), 193--220.

\bibitem{Kerr:2020ab}
D.~Kerr,
  {\em Dimension, comparison, and almost finiteness},
  J. Eur. Math. Soc. (JEMS) 22, 11 (2020), 3697--3745.

\bibitem{Kerr:2020aa}
D.~Kerr and G.~Szab\'{o},
  {\em Almost finiteness and the small boundary property},
  Comm. Math. Phys. 374, 1 (2020), 1--31.

\bibitem{Kirchberg:2000kq}
E.~Kirchberg and N.~C. Phillips,
  {\em Embedding of exact {$\mathrm{C}^*$}-algebras in the {C}untz algebra
  {$\mathcal{O}_2$}},
  J. Reine Angew. Math. 525\/ (2000), 17--53.

\bibitem{Kirchberg:2004qy}
E.~Kirchberg and W.~Winter,
  {\em Covering dimension and quasidiagonality},
  Internat. J. Math. 15, 1 (2004), 63--85.

\bibitem{Kumjian:1998nr}
A.~Kumjian, D.~Pask, and I.~Raeburn,
  {\em Cuntz-{K}rieger algebras of directed graphs},
  Pacific J. Math. 184, 1 (1998), 161--174.

\bibitem{Lazar:1971kx}
A.~J. Lazar and J.~Lindenstrauss,
  {\em Banach spaces whose duals are {$L\sb{1}$} spaces and their
  representing matrices},
  Acta Math. 126\/ (1971), 165--193.

\bibitem{Lin:2014aa}
H.~Lin and Z.~Niu,
  {\em Homomorphisms into simple {$\mathcal{Z}$}-stable {$\mathrm{C}^*$}-algebras},
  J. Operator Theory 71, 2 (2014), 517--569.

\bibitem{Matui:2011uq}
H.~Matui,
  {\em Classification of homomorphisms into simple {$\mathcal{Z}$}-stable
  {$\mathrm{C}^*$}-algebras},
  J. Funct. Anal. 260, 3 (2011), 797--831.

\bibitem{Murray:1943le}
F.~J. Murray and J.~von Neumann,
  {\em On rings of operators {IV}},
  Ann. of Math. (2) 44\/ (1943), 716--808.

\bibitem{Ng:2006kq}
P.~W. Ng and W.~Winter,
  {\em A note on subhomogeneous {$\mathrm{C}^*$}-algebras},
  C. R. Math. Acad. Sci. Soc. R. Can. 28, 3 (2006), 91--96.

\bibitem{Norris:1998wc}
J.~R. Norris,
  {\em Markov chains},
  vol.~2 of Cambridge Series in Statistical
  and Probabilistic Mathematics, Cambridge University Press, Cambridge, 1998.

\bibitem{Oxtoby:1941aa}
J.~C. Oxtoby and S.~M. Ulam,
  {\em Measure-preserving homeomorphisms and metrical transitivity},
  Ann. of Math. (2) 42\/ (1941), 874--920.

\bibitem{Pedersen:1969kq}
G.~K. Pedersen,
  {\em Measure theory for {$\mathrm{C}^*$} algebras {III}},
  Math. Scand. 25\/ (1969), 71--93.

\bibitem{Pedersen:1979rt}
G.~K. Pedersen,
  {\em {$\mathrm{C}^*$}-algebras and their automorphism groups}, vol.~14 of
London Mathematical Society Monographs,
  Academic Press Inc., London, 1979.

\bibitem{Phelps:2001rz}
R.~R. Phelps,
  {\em Lectures on {C}hoquet's theorem},
  second~ed., vol.~1757 of Lecture Notes in Mathematics,
  Springer-Verlag, Berlin, 2001.

\bibitem{Phillips:2000fj}
N.~C. Phillips,
  {\em A classification theorem for nuclear purely infinite simple {$\mathrm{C}^*$}-algebras},
  Doc. Math. 5\/ (2000), 49--114 (electronic).

\bibitem{Phillips:2001zm}
N.~C. Phillips,
  {\em Recursive subhomogeneous algebras},
  Trans. Amer. Math. Soc. 359, 10 (2007), 4595--4623 (electronic).

\bibitem{Robert:2010qy}
L.~Robert,
  {\em Classification of inductive limits of 1-dimensional {NCCW} complexes},
  Adv. Math. 231, 5 (2012), 2802--2836.

\bibitem{Robert:2010rz}
L.~Robert and L.~Santiago,
  {\em Classification of {$\mathrm{C}^*$}-homomorphisms from {$C_0(0,1]$} to a
  {$\mathrm{C}^*$}-algebra},
  J. Funct. Anal. 258, 3 (2010), 869--892.

\bibitem{Rordam:1998aa}
M.~R{\o}rdam,
  {\em On sums of finite projections},
  in: Operator algebras and operator theory ({S}hanghai, 1997),
  vol.~228 of Contemp. Math., Amer. Math. Soc., Providence, RI, 1998,
  pp.~327--340.

\bibitem{Rordam:2003rz}
M.~R{\o}rdam,
  {\em A simple {$\mathrm{C}^*$}-algebra with a finite and an infinite projection},
  Acta Math. 191, 1 (2003), 109--142.

\bibitem{Rordam:2004kq}
M.~R{\o}rdam,
  {\em The stable and the real rank of {$\mathcal{Z}$}-absorbing
  {$\mathrm{C}^*$}-algebras},
  Internat. J. Math. 15, 10 (2004), 1065--1084.

\bibitem{Rordam:2000fk}
M.~R{\o}rdam, F.~Larsen, and N.~Laustsen,
  {\em An introduction to {$K$}-theory for {$\mathrm{C}^*$}-algebras},
  vol.~49 of London Mathematical Society Student Texts,
  Cambridge University Press, Cambridge, 2000.

\bibitem{Rordam:2002yu}
M.~R{\o}rdam and E.~St{\o}rmer,
  {\em Classification of nuclear {$\mathrm{C}^*$}-algebras. {E}ntropy in
  operator algebras},
   volume 126 of Encyclopaedia of Mathematical
  Sciences,
  Springer-Verlag, Berlin, 2002,
  Operator Algebras and Non-commutative Geometry, 7.

\bibitem{Rosenberg:1987fp}
J.~Rosenberg and C.~Schochet,
  {\em The {K}{\"u}nneth theorem and the universal coefficient theorem for
  {K}asparov's generalized {$K$}-functor},
  Duke Math. J. 55, 2 (1987), 431--474.

\bibitem{Schafhauser:2020vs}
C.~Schafhauser,
  {\em A new proof of the {T}ikuisis--{W}hite--{W}inter theorem},
  J. Reine Angew. Math. 759\/ (2020), 291--304.

\bibitem{Schochet:1982vp}
C.~Schochet,
  {\em Topological methods for {$\mathrm{C}^*$}-algebras {II}. {G}eometric
  resolutions and the {K}\"{u}nneth formula},
  Pacific J. Math. 98, 2 (1982), 443--458.

\bibitem{Skandalis:1988aa}
G.~Skandalis,
  {\em Une notion de nucl\'{e}arit\'{e} en {$K$}-th\'{e}orie (d'apr\`es {J}.
  {C}untz)},
  $K$-Theory 1, 6 (1988), 549--573.

\bibitem{Skoufranis:2013aa}
P.~Skoufranis,
  {\em Closed unitary and similarity orbits of normal operators in purely
  infinite {$\mathrm{C}^*$}-algebras},
  J. Funct. Anal. 265, 3 (2013), 474--506.

\bibitem{Strung:2021aa}
K.~R. Strung,
  {\em An introduction to {$\mathrm{C}^*$}-algebras and the classification
  program},
  Advanced Courses in Mathematics, CRM Barcelona,
  Birkh\"{a}user/Springer, Cham, 2021,
  edited and with a foreword by Francesc Perera.

\bibitem{Thoma:1964aa}
E.~Thoma,
  {\em \"{U}ber unit\"{a}re {D}arstellungen abz\"{a}hlbarer, diskreter
  {G}ruppen},
  Math. Ann. 153\/ (1964), 111--138.

\bibitem{Thomsen:1992kq}
K.~Thomsen,
  {\em Inductive limits of interval algebras: unitary orbits of positive
  elements},
  Math. Ann. 293, 1 (1992), 47--63.

\bibitem{Thomsen:1994qy}
K.~Thomsen,
  {\em Inductive limits of interval algebras: the tracial state space},
  Amer. J. Math. 116, 3 (1994), 605--620.

\bibitem{Thomsen:1995qf}
K.~Thomsen,
  {\em Traces, unitary characters and crossed products by {${\mathbb{Z}}$}},
  Publ. Res. Inst. Math. Sci. 31, 6 (1995), 1011--1029.

\bibitem{Tikuisis:aa}
A.~Tikuisis, S.~White, and W.~Winter,
  {\em Quasidiagonality of nuclear {$\mathrm{C}^*$}-algebras},
  Ann. of Math. (2) 185, 1 (2017), 229--284.

\bibitem{Toms:2011fk}
A.~Toms,
  {\em K-theoretic rigidity and slow dimension growth},
  Invent. Math. 183, 2 (2011), 225--244.

\bibitem{Toms:2008vn}
A.~S. Toms,
  {\em On the classification problem for nuclear {$\mathrm{C}^*$}-algebras},
  Ann. of Math. (2) 167, 3 (2008), 1029--1044.

\bibitem{Tu:1999aa}
J.-L. Tu,
  {\em La conjecture de {B}aum-{C}onnes pour les feuilletages moyennables},
  $K$-Theory 17, 3 (1999), 215--264.

\bibitem{Villadsen:1998ys}
J.~Villadsen,
  {\em Simple {$\mathrm{C}^*$}-algebras with perforation},
  J. Funct. Anal. 154, 1 (1998), 110--116.

\bibitem{Villani:2009aa}
C.~Villani,
  {\em Optimal transport, old and new},
  vol.~338 of Grundlehren der
  Mathematischen Wissenschaften,
  Springer-Verlag, Berlin, 2009.

\bibitem{Wassermann:1994rz}
S.~Wassermann,
  {\em Exact {$\mathrm{C}^*$}-algebras and related topics},
  vol.~19 of Lecture Notes Series,
  Seoul National University Research Institute of Mathematics Global
  Analysis Research Center, Seoul, 1994.

\bibitem{Weyl:1912aa}
H.~Weyl,
  {\em Das asymptotische {V}erteilungsgesetz der {E}igenwerte linearer
  partieller {D}ifferentialgleichungen (mit einer {A}nwendung auf die {T}heorie
  der {H}ohlraumstrahlung)},
  Math. Ann. 71, 4 (1912), 441--479.

\bibitem{Winter:2012pi}
W.~Winter,
  {\em Nuclear dimension and {$\mathcal{Z}$}-stability of pure
  {$\mathrm{C}^*$}-algebras},
  Invent. Math. 187, 2 (2012), 259--342.

\bibitem{Winter:2010dn}
W.~Winter and J.~Zacharias,
  {\em The nuclear dimension of {$\mathrm{C}^*$}-algebras},
  Adv. Math. 224, 2 (2010), 461--498.

\end{thebibliography}
\end{document}